\documentclass{article}%
\usepackage{graphicx}
\usepackage{amsmath}
\usepackage{amsfonts}
\usepackage{amssymb}
\usepackage{setspace}
\usepackage[hmargin=3.5cm,vmargin=3.5cm]{geometry}%
\providecommand{\U}[1]{\protect\rule{.1in}{.1in}}
\providecommand{\U}[1]{\protect\rule{.1in}{.1in}}
\providecommand{\U}[1]{\protect\rule{.1in}{.1in}}
\providecommand{\U}[1]{\protect\rule{.1in}{.1in}}
\providecommand{\U}[1]{\protect\rule{.1in}{.1in}}
\providecommand{\U}[1]{\protect\rule{.1in}{.1in}}
\providecommand{\U}[1]{\protect\rule{.1in}{.1in}}
\providecommand{\U}[1]{\protect\rule{.1in}{.1in}}
\providecommand{\U}[1]{\protect\rule{.1in}{.1in}}
\providecommand{\U}[1]{\protect\rule{.1in}{.1in}}
\providecommand{\U}[1]{\protect\rule{.1in}{.1in}}
\providecommand{\U}[1]{\protect\rule{.1in}{.1in}}
\providecommand{\U}[1]{\protect\rule{.1in}{.1in}}
\providecommand{\U}[1]{\protect\rule{.1in}{.1in}}
\providecommand{\U}[1]{\protect\rule{.1in}{.1in}}
\providecommand{\U}[1]{\protect\rule{.1in}{.1in}}
\providecommand{\U}[1]{\protect\rule{.1in}{.1in}}
\providecommand{\U}[1]{\protect\rule{.1in}{.1in}}
\providecommand{\U}[1]{\protect\rule{.1in}{.1in}}
\providecommand{\U}[1]{\protect\rule{.1in}{.1in}}
\providecommand{\U}[1]{\protect\rule{.1in}{.1in}}
\providecommand{\U}[1]{\protect\rule{.1in}{.1in}}
\providecommand{\U}[1]{\protect\rule{.1in}{.1in}}
\providecommand{\U}[1]{\protect\rule{.1in}{.1in}}
\providecommand{\U}[1]{\protect\rule{.1in}{.1in}}
\providecommand{\U}[1]{\protect\rule{.1in}{.1in}}
\providecommand{\U}[1]{\protect\rule{.1in}{.1in}}
\providecommand{\U}[1]{\protect\rule{.1in}{.1in}}
\providecommand{\U}[1]{\protect\rule{.1in}{.1in}}
\providecommand{\U}[1]{\protect\rule{.1in}{.1in}}
\providecommand{\U}[1]{\protect\rule{.1in}{.1in}}
\providecommand{\U}[1]{\protect\rule{.1in}{.1in}}
\providecommand{\U}[1]{\protect\rule{.1in}{.1in}}
\providecommand{\U}[1]{\protect\rule{.1in}{.1in}}
\providecommand{\U}[1]{\protect\rule{.1in}{.1in}}
\providecommand{\U}[1]{\protect\rule{.1in}{.1in}}
\providecommand{\U}[1]{\protect\rule{.1in}{.1in}}
\providecommand{\U}[1]{\protect\rule{.1in}{.1in}}

\newtheorem{theorem}{Theorem}
{}

\newtheorem{example}{Example}

\newtheorem{lemma}{Lemma}
{}

\newtheorem{remark}{Remark}

\newenvironment{proof}[1][Proof]{\textbf{#1.} }{\ \rule{0.5em}{0.5em}}

\begin{document}

\title{On Explicit Estimations for the Bloch Eigenvalues of the One-dimensional
Schr\"{o}dinger Operator and the Kronig-Penney Model}
\author{Cemile Nur* and Oktay Veliev**\\Department of Basic Sciences, Yalova University\\Yalova, Turkey{\small .}\ {\small e-mail: cnur@yalova.edu.tr*}\\Department of Mechanical Engineering, Dogus University\\Istanbul, Turkey{\small .}\ {\small e-mail: oveliev@dogus.edu.tr**}}
\date{}
\maketitle

\begin{abstract}
In this paper, we consider the small and large eigenvalues of the
one-dimensional Schr\"{o}dinger operator $L(q)$ with a periodic, real and
locally integrable potential $q$. First we explicitly write out the first and
second terms of the asymptotic formulas for the large periodic and
antiperiodic eigenvalues and illustrate these formulas for the Kronig-Penney
model. Then we give estimates for the small periodic and antiperiodic
eigenvalues and for the length of the first gaps in the case of the
Kronig-Penney model. Moreover, we give error estimations and present a
numerical example.



\end{abstract}

\section{Introduction and Preliminary Facts}

In the present paper, we investigate the spectrum of the operator $L(q)$
generated in $L_{2}(-\infty,\infty)$ by the differential expression
\begin{equation}
l(y)=-y^{\prime\prime}+qy, \label{eq1}%
\end{equation}
where $q$ is $1$-periodic integrable on $[0,1]$ and a real-valued potential.
Without loss of generality, it is assumed that
\begin{equation}
\int_{0}^{1}q(x)dx=0. \label{eq2}%
\end{equation}

It is well known that (see for example~\cite{1, 3, 4, 5, 8, 10, 12}) the
spectrum $\sigma(L(q))$ of $L(q)$ is the union of the spectra $\sigma
(L_{t}(q))$ of the operators $L_{t}(q)$ for $t\in(-\pi,\pi]$ generated in
$L_{2}[0,1]$ by~\eqref{eq1} and the boundary conditions
\begin{equation}
y(1)=e^{it}y(0),\qquad y^{\prime}(1)=e^{it}y^{\prime}(0). \label{eq3}%
\end{equation}
Moreover, $\sigma(L(q))$ consists of the closed intervals whose end points are
the eigenvalues of $L_{t}(q)$ for $t=0,\pi$ (see~\eqref{eq3}). Therefore, to
study the spectrum of the self adjoint operator $L(q)$ it is enough to
investigate the eigenvalues of $L_{0}(q)$ and $L_{\pi}(q),$ which are called
the periodic and antiperiodic eigenvalues, respectively. The first periodic
eigenvalue is denoted by $\lambda_{0}$. The other eigenvalues of $L_{0}(q)$
and the eigenvalues of $L_{\pi}(q)$ are denoted by $\lambda_{n,j}$ and
$\mu_{n,j}$, respectively, for $n\in\mathbb{N}$, $j=1,2$, where $\mathbb{N}$
is the set of positive integers. Without loss of generality, it is assumed
that $\lambda_{n,1}\leq\lambda_{n,2}$ and $\mu_{n,1}\leq\mu_{n,2}$, for
$n\in\mathbb{N}$. It is known that (see~\cite{4}), the spectrum of the
Schr\"{o}dinger operator $L(q)$ consists of the real intervals
\[
\Gamma_{1}:=[\lambda_{0},\mu_{1,1}],\quad\Gamma_{2}:=[\mu_{1,2},\lambda
_{1,1}],\quad\Gamma_{3}:=[\lambda_{1,2},\mu_{2,1}],\quad\Gamma_{4}:=[\mu
_{2,2},\lambda_{2,1}],\ldots.
\]
The bands $\Gamma_{1}$, $\Gamma_{2}$, $\ldots$ of the spectrum $\sigma(L(q))$
of $L(q)$ are separated by the gaps
\[
\Delta_{1}:=(\mu_{1,1},\mu_{1,2}),\quad\Delta_{2}:=(\lambda_{1,1}%
,\lambda_{1,2}),\quad\Delta_{3}:=(\mu_{2,1},\mu_{2,2}),\quad\Delta
_{4}:=(\lambda_{2,1},\lambda_{2,2}),\ldots.
\]
For this reason, the investigation of the periodic and antiperiodic
eigenvalues is of great importance. In \cite{2} and~\cite{17}, Veliev
investigated the asymptotic behavior of large periodic and antiperiodic
eigenvalues and obtained asymptotic formulas for the lengths of the gaps in
the spectrum.

In this paper, we obtain more explicit formulas for the large eigenvalues of
$L_{0}(q)$ and $L_{\pi}(q)$ and estimate the small eigenvalues of these
operators. Moreover, we take the Kronig-Penney model as an example and
illustrate asymptotic formulas and estimations with this example. For this we
use some formulas from \cite{2, 13, 15, 16, 17}. For the independent reading
of this paper, first of all, we list the formulas of these papers that are
essentially used here.

In~\cite{2}, we established the following equalities
\begin{align}
&  (\lambda_{n,j}-(2\pi n)^{2}-A_{m}(\lambda_{n,j}))(\Psi_{n,j}(x),e^{i2\pi
nx})\nonumber\\
&  \qquad-(q_{2n}+B_{m}(\lambda_{n,j}))(\Psi_{n,j}(x),e^{-i2\pi nx}%
)=R_{m}(\lambda_{n,j}), \label{eq4}%
\end{align}
where
\[
A_{m}(\lambda_{n,j})=\sum_{k=1}^{m}a_{k}(\lambda_{n,j}),\quad B_{m}%
(\lambda_{n,j})=\sum_{k=1}^{m}b_{k}(\lambda_{n,j}),
\]%
\begin{equation}
a_{k}(\lambda_{n,j})=\sum_{n_{1},n_{2},...,n_{k}}\frac{q_{n_{1}}q_{n_{2}%
}...q_{n_{k}}q_{-n_{1}-n_{2}-...-n_{k}}}{\prod\limits_{s=1}^{k}(\lambda
_{n,j}-(2\pi(n-n_{1}-n_{2}-...-n_{s}))^{2})}, \label{eq5}%
\end{equation}%
\begin{equation}
b_{k}(\lambda_{n,j})=\sum_{n_{1},n_{2},...,n_{k}}\frac{q_{n_{1}}q_{n_{2}%
}...q_{n_{k}}q_{2n-n_{1}-n_{2}-...-n_{k}}}{\prod\limits_{s=1}^{k}%
(\lambda_{n,j}-(2\pi(n-n_{1}-n_{2}-...-n_{s}))^{2})}, \label{eq6}%
\end{equation}%
\begin{equation}
R_{m}(\lambda_{n,j})=\sum_{n_{1},n_{2},...,n_{m+1}}\frac{q_{n_{1}}q_{n_{2}%
}...q_{n_{m}}q_{n_{m+1}}(q(x)\Psi_{n,j}(x),e^{i2\pi(n-n_{1}-...-n_{m+1})x}%
)}{\prod\limits_{s=1}^{m+1}(\lambda_{n,j}-(2\pi(n-n_{1}-n_{2}-...-n_{s}%
))^{2})}. \label{eq7}%
\end{equation}
The summations in these formulas are taken over the indices satisfying the
conditions%
\[
n_{s}\neq0,\quad n_{1}+n_{2}+...+n_{s}\neq0,2n,
\]
for $s=1,2,...,m+1$. Moreover, it was proved that
\[
a_{k}(\lambda_{n,j})=O((\frac{\ln\left\vert n\right\vert }{n})^{k}),\qquad
b_{k}(\lambda_{n,j})=O((\frac{\ln\left\vert n\right\vert }{n})^{k})
\]
for $k\geq1$ and
\[
R_{m}(\lambda_{n,j})=O((\frac{\ln\left\vert n\right\vert }{n})^{m+1}).
\]

In~\cite{17}, the terms $a_{1}(\lambda_{n,j})$ and $a_{2}(\lambda_{n,j})$ were
investigated in detail and the following estimations were proved:
\begin{equation}
a_{1}(\lambda_{n,j})=a_{1}((2\pi n)^{2})+o(n^{-2})=o(n^{-1}) \label{eq8}%
\end{equation}
and
\[
a_{2}(\lambda_{n,j})=a_{2}((2\pi n)^{2})+o(n^{-3}\ln n),
\]
where $a_{1}((2\pi n)^{2})$ and $a_{2}((2\pi n)^{2})$ are obtained from
$a_{1}(\lambda_{n,j})$ and $a_{2}(\lambda_{n,j})$ by changing $\lambda_{n,j}$
to $(2\pi n)^{2}$ in the corresponding formulas and $q\in L_{1}[0,1]$.

Similarly, in~\cite{13}, we proved that the following relations are valid:
\begin{align}
&  b_{1}(\lambda_{n,j})=b_{1,n}((2\pi n)^{2})+o\left(  n^{-2}\right)  ,\qquad
b_{2}(\lambda_{n,j})=o\left(  n^{-2}\right)  ,\nonumber\\
&  \qquad b_{1,n}((2\pi n)^{2})=2Q_{0}Q_{2n}-S_{2n}, \label{eq9}%
\end{align}
where $b_{1}((2\pi n)^{2})$ is obtained from $b_{1}(\lambda_{n,j})$ by
changing $\lambda_{n,j}$ to $(2\pi n)^{2}$ in~\eqref{eq6}, $Q_{k}=(Q,e^{2\pi
ikx})$\ and $S_{k}=(S,e^{2\pi ikx})$ are the Fourier coefficients of the
following functions:%
\begin{equation}
Q(x)=\int_{0}^{x}q(t)\,dt,\qquad S(x)=Q^{2}(x) \label{eq10}%
\end{equation}
(see Lemma 6 of~\cite{13}).

Using~\eqref{eq8} and these relations, in~\cite{17}, the following asymptotic
formulas were obtained. If there exists a positive number $\varepsilon$ such
that
\begin{equation}
\left\vert q_{k}\right\vert \geq\frac{\varepsilon}{k}, \label{eq11}%
\end{equation}
and
\begin{equation}
\left\vert q_{k}-S_{k}+2Q_{0}Q_{k}\right\vert \geq\varepsilon k^{-2},
\label{eq12}%
\end{equation}
for $k=2n>N,$ where $N$ is a sufficiently large positive integer, then the
periodic eigenvalues\ $\lambda_{n,j}$, for $n>N$ and $j=1,2$, are simple and
satisfy the asymptotic formulas
\begin{equation}
\lambda_{n,j}=(2\pi n)^{2}+(-1)^{j}\left\vert q_{2n}\right\vert +o(n^{-1})
\label{eq13}%
\end{equation}
and
\begin{equation}
\lambda_{n,j}=(2\pi n)^{2}+a_{1}((2\pi n)^{2})+a_{2}((2\pi n)^{2}%
)+(-1)^{j}\left\vert q_{2n}-S_{2n}+2Q_{0}Q_{2n}\right\vert +o(n^{-2}),
\label{eq14}%
\end{equation}
respectively (see pages 66 and 70 in \cite{17}). Similarly, the antiperiodic
eigenvalues\ $\mu_{n,j}$ satisfy the formulas obtained from~\eqref{eq13}
and~\eqref{eq14} by replacing $2n$ with $2n-1$ if~\eqref{eq11}
and~\eqref{eq12} are satisfied for $k=2n-1>N$. Moreover, from these formulas
we obtained the following formulas for the lengths $\left\vert \Delta
_{k}\right\vert $ of the gaps $\Delta_{k}$:%
\begin{equation}
\left\vert \Delta_{k}\right\vert =2\left\vert q_{k}\right\vert +o(k^{-1})
\label{eq15}%
\end{equation}
and
\begin{equation}
\left\vert \Delta_{k}\right\vert =2\left\vert q_{k}-S_{k}+2Q_{0}%
Q_{k}\right\vert +o(k^{-2}) \label{eq16}%
\end{equation}
without using the conditions~\eqref{eq11} and~\eqref{eq12} (see page 71 in
\cite{17}).

The superiority of formula~\eqref{eq15} over the well-known Titchmarsh
formula
\begin{equation}
\left\vert \Delta_{2n}\right\vert =2\left\vert q_{2n}\right\vert +O(n^{-1})
\label{eq17}%
\end{equation}
(see~\cite[Chapter 21]{14},~\cite[Chapter 3]{1}, and their references) for any
potential $q\in L_{1}[0,1]$ can be explained as follows. Generally speaking,
certain Fourier coefficients $q_{2n}$ of many functions in $L_{1}[0,1]$ can be
of order $n^{-1}$ in the sense that $q_{2n}=O(n^{-1})$ and $n^{-1}=O(q_{2n})$
when $n$ belongs to some subset of $\mathbb{N}$ (e.g., when the potential $q$
exhibits a jump discontinuity, as occurs in the well known Kronig-Penney
model). In this context, equation~\eqref{eq15} with the error term $o(n^{-1})$
explicitly isolates the term $2\left\vert q_{2n}\right\vert $ from the error
$o(n^{-1})$, whereas equation~\eqref{eq17} does not. In this sense,
$2\left\vert q_{2n}\right\vert $ can be seen as the first term of the
asymptotic formula for the gaps in $\sigma(L(q))$ with the potential $q\in
L_{1}[0,1].$ Similarly,~\eqref{eq16} gives us the second term $2\left\vert
q_{2n}-S_{2n}+2Q_{0}Q_{2n}\right\vert -2\left\vert q_{2n}\right\vert $ of the
asymptotic formula for the gaps in $\sigma(L(q))$.

Note that, in~\cite{17}, we obtained asymptotic formulas with the error term
$O\left(  (\frac{\ln n}{n})^{k}\right)  $ for the $n$th\ periodic and
antiperiodic eigenvalues and the $n$th gaps in the spectrum of the operator
$L(q)$ with $q\in L_{1}[0,1]$, where $k=3,4,...$. Many studies have been
devoted to such asymptotic formulas for differentiable potentials
(see~\cite[Chap. 3]{1},~\cite[Chap. 1]{9} and~\cite[Chap. 2]{11} and their
references), whereas we obtain asymptotic formulas of arbitrary order for the
potential $q$ from $L_{1}[0,1]$. In Section 2, we improve the asymptotic
formulas obtained in~\cite{17} and give explicit asymptotic formulas for the
large eigenvalues.

In Section 3, we consider the small periodic and antiperiodic eigenvalues of
the Schr\"{o}dinger operator $L(q)$ and estimate the length of the gaps
$\Delta_{n}$ for the small values of $n$ in the Kronig-Penney model. We note
that, asymptotic formulas obtained in Section 2 cannot be used for estimating
the small eigenvalues. To give estimates for the small eigenvalues, we use
completely different methods than those of Section 2. In Section 3, we find
conditions on the Kronig-Penney potential for which the iteration
formula~\eqref{eq4} is also valid for the small eigenvalues.

Finally, note that the Kronig-Penney model is a simplified approach to
understanding the band structure of materials, it provides a foundational
understanding of band theory and is a stepping stone to more complex models in
solid-state physics. This model is particularly important in explaining the
quantum mechanical basis for electrical conduction in solids and the
properties of semiconductors.

\section{Asymptotic Formulas for the Large Eigenvalues}

To obtain explicit formulas for the large eigenvalues from~\eqref{eq14}, we
use the following lemmas.

\begin{lemma}
\label{l1} If $q\in L_{1}[0,1]$, then the following formula holds:
\begin{equation}
a_{1}((2\pi n)^{2})=\frac{\left\vert q_{2n}\right\vert ^{2}}{-16\pi^{2}n^{2}%
}+D(2n), \label{eq18}%
\end{equation}
where%
\[
D(k)=:\frac{i}{2\pi k}\int_{0}^{1}q(x)(Q(x,k)-Q_{k,0}))e^{-i2\pi kx}\,dx,
\]%
\[
Q(x,k)=\int_{0}^{x}q(t)e^{i2\pi kt}\,dt-q_{-k}x,\text{ }Q_{k,0}=\int_{0}%
^{1}Q(x,k)dx.
\]

\end{lemma}

\begin{proof}
Since $q_{-k}=\overline{q_{k}}$, $q_{0}=0$, from the definition of
$a_{1}((2\pi n)^{2})$ (see introduction), we obtain
\[
a_{1}((2\pi n)^{2})=\sum_{\substack{k=-\infty\\k\neq0,2n}}^{\infty}%
\frac{\left\vert q_{k}\right\vert ^{2}}{2\pi k(4\pi n-2\pi k)}.
\]
Using the equality
\begin{equation}
\frac{1}{2\pi k(4\pi n-2\pi k)}=\frac{1}{4\pi n}\left(  \frac{1}{2\pi k}%
+\frac{1}{4\pi n-2\pi k}\right)  , \label{eq19}%
\end{equation}
we see that
\begin{equation}
a_{1}((2\pi n)^{2})=\frac{1}{4\pi n}\sum_{\substack{k=-\infty\\k\neq
0,2n}}^{\infty}\frac{\left\vert q_{k}\right\vert ^{2}}{2\pi k}+\frac{1}{4\pi
n}\sum_{\substack{k=-\infty\\k\neq0,2n}}^{\infty}\frac{\left\vert
q_{k}\right\vert ^{2}}{(4\pi n-2\pi k)}. \label{eq20}%
\end{equation}
Since
\[
\frac{\left\vert q_{k}\right\vert ^{2}}{2\pi k}+\frac{\left\vert
q_{-k}\right\vert ^{2}}{2\pi(-k)}=0,
\]
we have the following equality for the first term on the right side
of~\eqref{eq20}:
\begin{equation}
\frac{1}{4\pi n}\sum_{\substack{k=-\infty\\k\neq0,2n}}^{\infty}\frac
{\left\vert q_{k}\right\vert ^{2}}{2\pi k}=\frac{\left\vert q_{2n}\right\vert
^{2}}{-16\pi^{2}n^{2}}. \label{eq21}%
\end{equation}

Now let us estimate the second term on the right side of~\eqref{eq20}. For
this, we consider the Fourier coefficients of $Q(x,2n)$. It is clear that
\[
Q(1,2n)=Q(0,2n)=0,\quad Q^{\prime}(x,2n)=q(x)e^{i4\pi nx}\,-q_{-2n}.
\]
Therefore, using the integration by parts, we see that the Fourier
coefficients $Q_{2n,2n-k}=(Q(x,2n),e^{i2\pi(2n-k)x})$ of $Q(x,2n)$ with
respect to the orthonormal basis $\left\{  e^{2\pi ikx}:k\in\mathbb{Z}%
\right\}  $, for $2n-k\neq0$, are%
\[
Q_{2n,2n-k}=\frac{q_{-k}}{i(4\pi n-2\pi k)}.
\]
Moreover, $q_{0}=0$ according to~\eqref{eq2}. Thus, the Fourier decomposition
of $Q(x,2n)$ has the form
\begin{equation}
Q(x,2n)=Q_{2n,0}+\sum_{\substack{k=-\infty\\k\neq0,2n}}^{\infty}\frac
{q_{-k}e^{i2\pi(2n-k)x}\,}{i(4\pi n-2\pi k)}. \label{eq22}%
\end{equation}
Thus,~\eqref{eq18} follows from~\eqref{eq20}-\eqref{eq22}. The lemma is proved.
\end{proof}

Now we consider $a_{2}((2\pi n)^{2})$. It is clear that
\begin{align*}
&  a_{2}((2\pi n)^{2})=\sum\limits_{k,k+l\neq0,2n}\frac{q_{k}q_{l}q_{-k-l}%
}{[(2\pi n)^{2}-(2\pi(n-k))^{2}][(2\pi n)^{2}-(2\pi(n-k-l))^{2}]}\\
&  \qquad=\sum_{k,l}\frac{q_{k}q_{l}q_{-k-l}}{2\pi k(4\pi n-2\pi k)(2\pi
k+2\pi l)(4\pi n-2\pi k-2\pi l)}.
\end{align*}
Using~\eqref{eq19} and
\[
\frac{1}{(k+l)(2n-k-l)}=\frac{1}{2n}\left(  \frac{1}{k+l}+\frac{1}%
{2n-k-l}\right)  ,
\]
we obtain that
\begin{equation}
a_{2}((2\pi n)^{2})=\frac{1}{16\pi^{2}n^{2}}(I_{1}+I_{2}+I_{3}+I_{4}),
\label{eq23}%
\end{equation}
where
\[
I_{1}(q)=\sum_{k,l}\frac{q_{k}q_{l}q_{-k-l}}{2\pi k(2\pi k+2\pi l)},
\]%
\[
I_{2}(n,q)=\sum_{k,l}\frac{q_{k}q_{l}q_{-k-l}}{(4\pi n-2\pi k)(2\pi k+2\pi
l)},
\]%
\[
I_{3}(n,q)=\sum_{k,l}\frac{q_{k}q_{l}q_{-k-l}}{2\pi k(4\pi n-2\pi k-2\pi l)},
\]
and%
\[
I_{4}(n,q)=\sum_{k,l}\frac{q_{k}q_{l}q_{-k-l}}{(4\pi n-2\pi k)(4\pi n-2\pi
k-2\pi l)}.
\]
Now we estimate $I_{1}(q)$ and $I_{j}(n,q)$, for $j=2,3,4$, in the following lemmas.

\begin{lemma}
\label{l2} The equality $I_{1}(q)=0$ holds.
\end{lemma}

\begin{proof}
We prove this lemma in two ways. First grouping the terms
\begin{align*}
&  \frac{q_{k}q_{l}q_{-k-l}}{4\pi^{2}k(k+l)},\quad\frac{q_{l}q_{k}q_{-k-l}%
}{4\pi^{2}l(k+l)},\quad\frac{q_{k}q_{-k-l}q_{l}}{4\pi^{2}k(-l)},\\
&  \frac{q_{l}q_{-k-l}q_{k}}{4\pi^{2}l(-k)},\quad\frac{q_{-k-l}q_{l}q_{k}%
}{4\pi^{2}(-k-l)(-k)},\quad\frac{q_{-k-l}q_{k}q_{l}}{4\pi^{2}(-k-l)(-l)}%
\end{align*}
with the equal multiplicands,%
\begin{align*}
&  q_{k}q_{l}q_{-k-l},\quad q_{l}q_{k}q_{-k-l},\quad q_{k}q_{-k-l}q_{l},\\
&  q_{l}q_{-k-l}q_{k},\quad q_{-k-l}q_{l}q_{k},\quad q_{-k-l}q_{k}q_{l}%
\end{align*}
and using the equality
\[
\frac{1}{k(k+l)}+\frac{1}{l(k+l)}=\frac{1}{kl}%
\]
we obtain the proof of the lemma.

The second way is the following. Using the integration by parts one can easily
verify that the Fourier coefficients $Q_{k}=(Q,\,e^{2\pi ikx})\ $of$\,Q$
(see~\eqref{eq10}) with respect to the orthonormal basis $\left\{  e^{2\pi
ikx}:k\in\mathbb{Z}\right\}  $, for $k\neq0$, are $\dfrac{q_{k}}{-2\pi ki}$.
Therefore, we have
\begin{equation}
Q(x)-Q_{0}=%
{\textstyle\sum\limits_{k\neq0}}
\frac{q_{k}}{-2\pi ki}e^{2\pi ikx}. \label{eq24}%
\end{equation}
Using this and~\eqref{eq2}, it is easy to verify that
\[
I_{1}=\int_{0}^{1}q(x)(Q(x)-Q_{0})^{2}\,dx.
\]
On the other hand, $(Q(x)-Q_{0})^{\prime}=q(x)$, $Q(1)=Q(0)=0$, and hence
\begin{equation}
\int_{0}^{1}q(x)(Q(x)-Q_{0})^{k}\,dx=\frac{1}{k+1}((Q(1)-Q_{0})^{k+1}%
-(Q(0)-Q_{0})^{k+1})=0, \label{eq25}%
\end{equation}
for $k=1,2,...$. The lemma is proved.
\end{proof}

\begin{lemma}
\label{l3} The equalities $I_{j}(n,q)=o(n^{2})$ for $j=2,3,4$ hold.
\end{lemma}

\begin{proof}
Using~\eqref{eq22} and~\eqref{eq24}, we obtain
\begin{equation}
I_{2}(n,q)=\int_{0}^{1}q(x)\left(  Q(x)-Q_{0}\right)  \left(  Q(x,2n)-Q_{2n,0}%
\right)  \,dx. \label{eq26}%
\end{equation}
Formula~\eqref{eq22} can be written in the form
\begin{equation}
Q(x,2n)-Q_{2n,0}=\sum_{\substack{l=-\infty\\l+k\neq0,2n}}^{\infty}%
\frac{q_{-k-l}e^{i2\pi(2n-k-l)x}}{i2\pi(2n-k-l)}. \label{eq27}%
\end{equation}
Using these notations, we obtain
\begin{equation}
I_{3}(n,q)=\int_{0}^{1}q(x)\left(  Q(x)-Q_{0}\right)  \left(  Q(x,2n)-Q_{2n,0}%
\right)  \,dx=I_{2}(n,q). \label{eq28}%
\end{equation}
In the same way, we establish that
\begin{equation}
I_{4}(n,q)=\int_{0}^{1}q(x)\left(  Q(x,2n)-Q_{2n,0}\right)  ^{2}\,dx.
\label{eq29}%
\end{equation}
Now the proof of the lemma follows from~\eqref{eq2},~\eqref{eq25}-\eqref{eq29}
and the following obvious statements: $Q(x,2n)$ converges uniformly, for
$x\in\lbrack0,1]$, to zero as $n\rightarrow\infty$ (see page 57 of~\cite{17}).
The lemma is proved.
\end{proof}

Now from~\eqref{eq14},~\eqref{eq18},~\eqref{eq23} and Lemmas~\ref{l1}%
-\ref{l3}, we obtain the following asymptotic formulas for the periodic eigenvalues:

\begin{theorem}
\label{t1} If~\eqref{eq12}, for $k=2n$, holds, then the eigenvalues
$\lambda_{n,j}$, for $n>N$ and $j=1,2$, are simple and satisfy the asymptotic
formula
\begin{equation}
\lambda_{n,j}=(2\pi n)^{2}+D(2n)+(-1)^{j}\left\vert q_{2n}-S_{2n}+2Q_{0}%
Q_{2n}\right\vert +o(n^{-2}) \label{eq30}%
\end{equation}
as $n\rightarrow\infty$, where $D(2n)$ is defined in Lemma~\ref{l1}.
\end{theorem}

The proof of the corresponding results for the antiperiodic problem can be
carried out in a similar way.

\begin{theorem}
\label{t2} If~\eqref{eq12}, for $k=2n-1$, holds, then the eigenvalues\ $\mu
_{n,j}$ of the operator $L_{\pi}(q)$, for $n>N$ and $j=1,2$, are simple and
satisfy the asymptotic formula
\begin{equation}
\mu_{n,j}=(2\pi n-\pi)^{2}+D(2n-1)+(-1)^{j}\left\vert q_{2n-1}-S_{2n-1}%
+2Q_{0}Q_{2n-1}\right\vert +o(n^{-2}), \label{eq31}%
\end{equation}
as $n\rightarrow\infty$.
\end{theorem}

Now, we take the Kronig-Penney model as an example and use this example to
illustrate the obtained asymptotic formulas~\eqref{eq30} and~\eqref{eq31}. The
Kronig-Penney model has been studied in many works (see, for example,~\cite{7}%
, ~\cite[Chapter 3]{1} and~\cite[Chapter 21]{14}). In this case, the potential
$q(x)$ has the form
\[
q(x)=%
\begin{cases}
a & \text{if }x\in\lbrack0,c]\\[2mm]%
b & \text{if }x\in(c,d],
\end{cases}
\]
and $q(x+d)=q(x)$, where $c\in(0,d)$. For simplicity of notation and without
loss of generality, we assume that $d=1$, $a<b$ and~\eqref{eq2} is satisfied.
Then, we have
\begin{equation}
q(x)=%
\begin{cases}
a & \text{if }x\in\lbrack0,c]\\[2mm]%
b & \text{if }x\in(c,1],
\end{cases}
\qquad\label{eq32}%
\end{equation}
where $a<0<b$ and%
\begin{equation}
ac+(1-c)b=0. \label{eq33}%
\end{equation}
To estimate the second term of the asymptotic formulas for the periodic
eigenvalues when $q$ is defined by formula~\eqref{eq32}, we need to calculate
the term $D(2n)$ in~\eqref{eq30}, since the terms $q_{k}$, $Q_{k}$, $S_{k}%
$\ were calculated in~\cite{17} (see pages 76-78 in~\cite{17}) and the
following equalities were obtained:
\begin{equation}
q_{k}=\frac{a-b}{2\pi ki}(1-e^{-2\pi kic}), \label{eq34}%
\end{equation}%
\begin{equation}
Q_{k}=\frac{q_{k}}{2\pi ki}=\frac{a-b}{(2\pi k)^{2}}(e^{-2\pi kic}-1),\qquad
Q_{0}=\frac{1}{2}b(c-1), \label{eq35}%
\end{equation}%
\begin{align}
&  S_{k}=\frac{a^{2}}{\pi ki}\biggl(\frac{e^{-2\pi ikc}-1}{(2\pi k)^{2}%
}\,-\frac{ce^{-2\pi ikc}}{2\pi ik}\biggr)\nonumber\\
&  \qquad+\frac{b^{2}}{\pi ki}\biggl(\frac{1-e^{-2\pi ikc}}{(2\pi k)^{2}%
}+\frac{ce^{-2\pi ikc}-1}{2\pi ik}\biggr)-\frac{b^{2}}{\pi ki}\biggl(\frac
{e^{-2\pi ikc}-1}{2\pi ik}\biggr). \label{eq36}%
\end{align}
Using the definition of $Q(x,k)$ (see Lemma~\ref{l1}), by direct calculations
we obtain
\begin{equation}
Q(x,k)=%
\begin{cases}
\dfrac{a}{i2\pi k}(e^{i2\pi kx}-1)-q_{-k}x & \text{if }x\in\lbrack0,c]\\[2mm]%
\dfrac{b}{i2\pi k}(e^{i2\pi kx}-1)-q_{-k}x+q_{-k} & \text{if }x\in(c,1].
\end{cases}
\label{eq37}%
\end{equation}
Therefore, we have
\begin{align}
Q_{k,0}  &  =\int_{0}^{c}\bigl(\dfrac{a}{i2\pi k}(e^{i2\pi kx}-1)-q_{-k}%
x\bigr)dx+\int_{c}^{1}\bigl(\dfrac{b}{i2\pi k}(e^{i2\pi kx}-1)-q_{-k}%
x+q_{-k}\bigr)dx\nonumber\\
&  =\frac{(b-a)(e^{i2\pi kc}-1)}{4\pi^{2}k^{2}}+q_{-k}(\frac{1}{2}%
-c)=\frac{(b-a)(e^{i2\pi kc}-1)}{4\pi^{2}k^{2}}+\frac{(a-b)(c-\frac{1}{2}%
)}{2\pi ki}(1-e^{2\pi kic})\nonumber\\
&  =\frac{(b-a)(e^{i2\pi kc}-1)}{4\pi^{2}k^{2}}+\frac{(a+b)(e^{i2\pi kc}%
-1)}{4\pi ki} \label{eq38}%
\end{align}
and
\begin{align}
&  D(k)=\frac{i}{2\pi k}\int_{0}^{1}q(x)Q(x,k)e^{-i2\pi kx}\,dx-\frac
{iQ_{k,0}}{2\pi k}\int_{0}^{1}q(x)e^{-i2\pi kx}\,dx\nonumber\\
&  \qquad=\frac{i}{2\pi k}\biggl(\int_{0}^{c}q(x)Q(x,k)e^{-i2\pi kx}%
\,dx+\int_{c}^{1}q(x)Q(x,k)e^{-i2\pi kx}\,dx\biggr)-\frac{iQ_{k,0}q_{k}}{2\pi
k}. \label{eq39}%
\end{align}
Using~\eqref{eq32},~\eqref{eq37}, and the integration by parts in the last
integrals give
\begin{align}
&  \int_{0}^{c}q(x)Q(x,k)e^{-i2\pi kx}\,dx=\int_{0}^{c}a\biggl(\dfrac{a}{i2\pi
k}(e^{i2\pi kx}-1)-q_{-k}x\biggr)e^{-i2\pi kx}\,dx\nonumber\\
&  \qquad=\dfrac{a^{2}}{i2\pi k}\int_{0}^{c}(1-e^{-i2\pi kx})\,dx-aq_{-k}%
\int_{0}^{c}xe^{-i2\pi kx}\,dx\nonumber\\
&  \qquad=\dfrac{a^{2}}{i2\pi k}\biggl(c+\dfrac{e^{-i2\pi kc}-1}{i2\pi
k}\biggr)+\dfrac{aq_{-k}}{i2\pi k}\biggl(ce^{-i2\pi kc}+\dfrac{e^{-i2\pi
kc}-1}{i2\pi k}\biggr) \label{eq40}%
\end{align}
and
\begin{align}
&  \int_{c}^{1}q(x)Q(x,k)e^{-i2\pi kx}\,dx=\int_{c}^{1}b\biggl(\dfrac{b}{i2\pi
k}(e^{i2\pi kx}-1)-q_{-k}x+q_{-k}\biggr)e^{-i2\pi kx}\,dx\nonumber\\
&  \qquad=\dfrac{b^{2}}{i2\pi k}\int_{c}^{1}(1-e^{-i2\pi kx})\,dx-bq_{-k}%
\int_{c}^{1}xe^{-i2\pi kx}\,dx+bq_{-k}\int_{c}^{1}e^{-i2\pi kx}\,dx\nonumber\\
&  \qquad=\dfrac{b^{2}}{i2\pi k}\biggl(1-c-\dfrac{e^{-i2\pi kc}-1}{i2\pi
k}\biggr)+\dfrac{bq_{-k}}{i2\pi k}\biggl(1-ce^{-i2\pi kc}-\dfrac{e^{-i2\pi
kc}-1}{i2\pi k}\biggr)+\dfrac{bq_{-k}}{i2\pi k}(e^{-i2\pi kc}-1). \label{eq41}%
\end{align}
Using~\eqref{eq38},~\eqref{eq40},~\eqref{eq41} and~\eqref{eq33}
in~\eqref{eq39} and doing some simplification, we obtain
\[
D(k)=\frac{-ab}{4\pi^{2}k^{2}}+O(\frac{1}{k^{3}}).
\]
Now using~\eqref{eq34}-\eqref{eq36} and~\eqref{eq33}, we obtain
\begin{align}
&  q_{k}-S_{k}+2Q_{0}Q_{k}=\frac{a-b}{2\pi ki}(1-e^{-2\pi kic})+\frac
{ba}{(2\pi k)^{2}}(e^{-2\pi ikc}+1)+O(k^{-3})\nonumber\\
&  \qquad=e^{-\pi kic}\biggl(\frac{a-b}{\pi k}\sin(\pi kc)+\frac{ba}{2(\pi
k)^{2}}\cos(\pi kc)\biggr)+O(k^{-3}). \label{eq42}%
\end{align}

In Theorem 2.1.11 of \cite{17}, we proved that if $c\in\mathbb{Q}$, then
$|q_{k}-S_{k}+2Q_{0}Q_{k}|>c_{0}k^{-2}$, for some positive constant $c_{0}$,
where $\mathbb{Q}$ is the set of rational numbers. Therefore the following
theorems follow from Theorem~\ref{t1} and Theorem~\ref{t2}, respectively.

\begin{theorem}
\label{t3} If $c\in\mathbb{Q}$, then the periodic eigenvalues $\lambda_{n,j}$,
for $n>N$ and $j=1,2$, are simple and satisfy the asymptotic formula
\begin{equation}
\lambda_{n,j}=(2\pi n)^{2}+\frac{-ab}{16\pi^{2}n^{2}}+(-1)^{j}\left\vert
q_{2n}-S_{2n}+2Q_{0}Q_{2n}\right\vert +o(n^{-2}) \label{eq43}%
\end{equation}
as $n\rightarrow\infty$.
\end{theorem}

and

\begin{theorem}
\label{t4} If $c\in\mathbb{Q}$, then the antiperiodic eigenvalues\ $\mu_{n,j}$
of the operator $L_{\pi}(q)$, for $n>N$ and $j=1,2$, are simple and satisfy
the asymptotic formula
\begin{equation}
\mu_{n,j}=(2\pi n-\pi)^{2}+\frac{-ab}{4\pi^{2}(2n-1)^{2}}+(-1)^{j}\left\vert
q_{2n-1}-S_{2n-1}+2Q_{0}Q_{2n-1}\right\vert +o(n^{-2}) \label{eq44}%
\end{equation}
as $n\rightarrow\infty$.
\end{theorem}

Now, let us consider the general case $c\in\mathbb{R}$. By~\eqref{eq42}, we
have
\[
|q_{k}-S_{k}+2Q_{0}Q_{k}|=|\alpha\sin(\pi kc)+\beta\cos(\pi kc)|+O(\frac
{1}{k^{3}}),
\]
where $\alpha=\dfrac{a-b}{\pi k}$ and $\beta=\dfrac{ab}{2\pi^{2}k^{2}}$. There
exists $\theta\in\lbrack0,2\pi)$ such that
\[
\cos\theta=\frac{\alpha}{\sqrt{\alpha^{2}+\beta^{2}}},\qquad\sin\theta
=\frac{\beta}{\sqrt{\alpha^{2}+\beta^{2}}},
\]
where $\sqrt{\alpha^{2}+\beta^{2}}=\dfrac{b-a}{\pi k}+O(\dfrac{1}{k^{3}})$.
Then, we have
\[
|q_{k}-S_{k}+2Q_{0}Q_{k}|=\frac{b-a}{\pi k}|\sin(\pi kc+\theta)|+O(\frac
{1}{k^{3}}),
\]
Therefore, if there exists $\varepsilon>0$ such that the inequality
\begin{equation}
|\sin(\pi kc+\theta)|>\frac{\varepsilon}{k} \label{eq45}%
\end{equation}
is satisfied, then~\eqref{eq12} holds. Thus the following theorems follow from
Theorem~\ref{t1} and Theorem~\ref{t2}, respectively.

\begin{theorem}
\label{t5} If~\eqref{eq45} is satisfied for $k=2n$, then the periodic
eigenvalues $\lambda_{n,j}$, for $n>N$ and $j=1,2$, are simple and satisfy
asymptotic formula~\eqref{eq43}.
\end{theorem}

and

\begin{theorem}
\label{t6} If~\eqref{eq45} is satisfied for $k=2n-1$, then the antiperiodic
eigenvalues\ $\mu_{n,j}$ of the operator $L_{\pi}(q)$, for $n>N$ and $j=1,2$,
are simple and satisfy asymptotic formula~\eqref{eq44}.
\end{theorem}

\section{On the Small Eigenvalues}

In this section, we consider the small periodic and antiperiodic eigenvalues
of the Schr\"{o}dinger operator $L(q)$ with potential~\eqref{eq32}. We shall
focus on the operator $L_{0}(q)$ with potential~\eqref{eq32}. The
investigation of $L_{\pi}(q)$ is similar.

It is known that~\cite{6}
\[
(2\pi n)^{2}-M\leq\lambda_{n,j}\leq(2\pi n)^{2}+M,
\]
for $n\geq1$, where $M=\max\{|a|,b\}$. Besides, $|\lambda_{0}|\leq M$.

If $k\neq\pm n$, then
\begin{align}
&  |\lambda_{n,j}-(2\pi k)^{2}|\geq|(2\pi n)^{2}-(2\pi k)^{2}|-M=4\pi
^{2}|n-k||n+k|-M\nonumber\\
&  \qquad\geq4\pi^{2}(2n-1)-M, \label{eq46}%
\end{align}
for $n\geq1$ and under the assumption $M\leq\dfrac{4\pi^{2}(2n-1)}{3}$. If
$n=1$, we have $|\lambda_{1,j}|\leq4\pi^{2}+M$ and
\[
|\lambda_{1,j}-(2\pi k)^{2}|\geq||\lambda_{1,j}|-(2\pi k)^{2}|\geq16\pi
^{2}-|\lambda_{1,j}|\geq12\pi^{2}-M,
\]
for $|k|\geq2$. Moreover, if $n\geq2$, we have $|\lambda_{n,j}|\geq
|\lambda_{2,j}|\geq16\pi^{2}-M$ and
\[
|\lambda_{n,j}-(2\pi k)^{2}|\geq||\lambda_{2,j}|-(2\pi k)^{2}|\geq
|\lambda_{2,j}|-4\pi^{2}\geq12\pi^{2}-M,
\]
for $k\neq\pm n$. The analogous inequalities can be written for the
antiperiodic eigenvalues, from the inequalities
\[
(2n-1)^{2}\pi^{2}-M\leq\mu_{n,j}\leq(2n-1)^{2}\pi^{2}+M,
\]
for $n\geq2$. If $k\neq\pm n$, then
\begin{align}
&  |\mu_{n,j}-(2k-1)^{2}\pi^{2}|\geq|(2n-1)^{2}\pi^{2}-(2k-1)^{2}\pi
^{2}|-M=4\pi^{2}|n-k||n+k-1|-M\nonumber\\
&  \qquad\geq4\pi^{2}(2n-2)-M, \label{eq47}%
\end{align}
for $n\geq2$, under the assumption $M\leq\dfrac{8\pi^{2}(n-1)}{3}$. If $n=1$,
we have $|\mu_{1,j}|\leq\pi^{2}+M$ and
\[
|\mu_{1,j}-(2k-1)^{2}\pi^{2}|\geq||\mu_{1,j}|-(2k-1)^{2}\pi^{2}|\geq9\pi
^{2}-|\mu_{1,j}|\geq8\pi^{2}-M,
\]
for $|k|\geq2$, and we assume $M\leq8\pi^{2}/3$, for $n=1$. Besides, if
$n\geq2$, we have $|\mu_{n,j}|\geq|\mu_{2,j}|\geq9\pi^{2}-M$ and
\[
|\mu_{n,j}-(2k-1)^{2}\pi^{2}|\geq||\mu_{2,j}|-(2k-1)^{2}\pi^{2}|\geq|\mu
_{2,j}|-\pi^{2}\geq8\pi^{2}-M,
\]
for $k\neq\pm n$.

We note that, the iteration formula~\eqref{eq4} was obtained in~\cite{2} and
has also been used in the previous section for the large eigenvalues to obtain
subtle asymptotic formulas. In this section, we find conditions on
potential~\eqref{eq32} for which the iteration formula~\eqref{eq4} is also
valid for the small eigenvalues, as $m$ tends to infinity. We also note that,
it is not easy to give such conditions, the calculations are very technical
and long. Since the potential $q$ is real, we have the following:
\begin{equation}
q_{-k}=\overline{q_{k}},\qquad v_{n,j}=\overline{u_{n,j}}, \label{eq48}%
\end{equation}
where $u_{n,j}=(\Psi_{n,j},e^{i2\pi nx})$ and $v_{n,j}=(\Psi_{n,j},e^{-i2\pi
nx})$ (see~(2.1.83) in~\cite{17}). Now, in order to give the main results of
this section, we prove the following lemmas. Without loss of generality, we
assume that $\Psi_{n,j}(x)$ is a normalized eigenfunction corresponding to the
eigenvalue $\lambda_{n,j}$. First, we consider the case $n\geq1$.

\begin{lemma}
\label{l4} If $M\leq\dfrac{4\pi^{2}(2n-1)}{3}$, where $M=\max\{|a|,b\}$, then
the statements

\textbf{(a)} $\lim_{m\rightarrow\infty}R_{m}(\lambda_{n,j})=0$

\textbf{(b)} $|u_{n,j}|>0$

are valid, where $R_{m}(\lambda_{n,j})$\ is defined by~\eqref{eq7}.

\begin{proof}
\textbf{ (a)} Considering the greatest summands of $R_{2m+1}(\lambda_{n,j})$
in absolute value and by~\eqref{eq46}, we obtain for $n=1$,
\begin{align*}
&  |R_{2m+1}(\lambda_{1,j})|<\frac{4^{m}|q_{1}||q_{-2}|^{m}|q_{2}|^{m}%
M}{|\lambda_{1,j}|^{m+1}|\lambda_{1,j}-16\pi^{2}|^{m}}+\frac{4^{m}%
|q_{-1}||q_{-2}|^{m}|q_{2}|^{m}M}{|\lambda_{1,j}|^{m}|\lambda_{1,j}-16\pi
^{2}|^{m+1}}+\frac{4^{2m}|q_{-1}|^{m+1}|q_{1}|^{m}M}{|\lambda_{1,j}-16\pi
^{2}|^{m+1}|\lambda_{1,j}-36\pi^{2}|^{m}}\\
&  \qquad<\frac{4^{m}(b-a)^{2m+1}M}{2^{2m}\pi^{2m+1}(4\pi^{2}-M)^{m+1}%
(12\pi^{2}-M)^{m}}+\frac{4^{m}(b-a)^{2m+1}M}{2^{2m}\pi^{2m+1}(4\pi^{2}%
-M)^{m}(12\pi^{2}-M)^{m+1}}\\
&  \qquad+\frac{4^{2m}(b-a)^{2m+1}M}{\pi^{2m+1}(12\pi^{2}-M)^{m+1}(32\pi
^{2}-M)^{m}}<\frac{5\pi}{3}(\frac{1}{4\pi^{2}})^{m}+\frac{\pi}{3}(\frac
{4}{23\pi^{2}})^{m},
\end{align*}
for $n=2$,
\begin{align*}
&  |R_{2m+1}(\lambda_{2,j})|<\frac{4^{m}|q_{1}||q_{-2}|^{m}|q_{2}|^{m}%
M}{|\lambda_{2,j}-4\pi^{2}|^{2m+1}}+\frac{4^{m}|q_{2}||q_{-1}|^{m}|q_{1}%
|^{m}M}{|\lambda_{2,j}-4\pi^{2}|^{m}|\lambda_{2,j}|^{m+1}}+\frac{4^{2m}%
|q_{-1}||q_{-2}|^{m}|q_{2}|^{m}M}{|\lambda_{2,j}-4\pi^{2}|^{m+1}|\lambda
_{2,j}-36\pi^{2}|^{m}}\\
&  \qquad<\frac{4^{m}(b-a)^{2m+1}M}{2^{2m}\pi^{2m+1}(12\pi^{2}-M)^{2m+1}%
}+\frac{4^{m}(b-a)^{2m+1}M}{2^{2m}\pi^{2m+1}(12\pi^{2}-M)^{m}(16\pi
^{2}-M)^{m+1}}\\
&  \qquad+\frac{4^{2m}(b-a)^{2m+1}M}{\pi^{2m+1}(12\pi^{2}-M)^{m+1}(20\pi
^{2}-M)^{m}}<4\pi(\frac{1}{\pi^{2}})^{m}+\frac{8\pi}{3}(\frac{2}{3\pi^{2}%
})^{m}+4\pi(\frac{8}{\pi^{2}})^{m},
\end{align*}
and in general, we have $|R_{2m+1}(\lambda_{n,j})|<dr^{m}$, for $n\geq3$, for
some constant $d>0$ and $0<r<1$. \ Therefore, $\lim_{m\rightarrow\infty}%
R_{m}(\lambda_{n,j})=0$. Similarly, we prove that $\lim_{m\rightarrow\infty
}R_{m}^{\ast}(\lambda_{n,j})=0$.

\textbf{(b)} Suppose the contrary, $u_{n,j}=0$. Since the system of root
functions $\{e^{2\pi ikx}:k\in\mathbb{Z}\}$\ of $L_{0}(0)$ forms an
orthonormal basis for $L_{2}[0,1]$ and by~\eqref{eq48}, we have the
decomposition
\[
\Psi_{n,j}=u_{n,j}e^{i2\pi nx}+\overline{u_{n,j}}e^{-i2\pi nx}+\sum
_{k\in\mathbb{Z},k\neq\pm n}\left(  \Psi_{n,j},e^{i2\pi kx}\right)  e^{i2\pi
kx}%
\]
for the normalized eigenfunction $\Psi_{n,j}$ corresponding to the eigenvalue
$\lambda_{n,j}$ of $L_{0}(q)$. By Parseval's equality, we obtain
\begin{equation}
\sum_{k\in\mathbb{Z},k\neq\pm n}|(\Psi_{n,j},e^{i2\pi kx})|^{2}=1.
\label{eq49}%
\end{equation}
First, we consider the case $n=1$. Using the relation
\[
(\lambda_{N,j}-(2\pi n)^{2})(\Psi_{N,j},e^{i2\pi nx})=(q\Psi_{N,j},e^{i2\pi
nx}),
\]
\eqref{eq46}, and the Bessel inequality, we obtain
\begin{align*}
&  \sum_{k\in\mathbb{Z},k\neq\pm1}|(\Psi_{1,j},e^{i2\pi kx})|^{2}=\sum
_{k\in\mathbb{Z},k\neq\pm1}\frac{|(q\Psi_{1,j},e^{i2\pi kx})|^{2}}%
{|\lambda_{1,j}-(2\pi k)^{2}|^{2}}\\
&  \qquad\leq\frac{1}{(4\pi^{2}-M)^{2}}\sum_{k\in\mathbb{Z},k\neq\pm1}%
|(q\Psi_{1,j},e^{i2\pi kx})|^{2}<\frac{M^{2}}{(4\pi^{2}-M)^{2}}<1,
\end{align*}
and in the case $n\geq2$, we have
\begin{align*}
&  \sum_{k\in\mathbb{Z},k\neq\pm n}|(\Psi_{n,j},e^{i2\pi kx})|^{2}=\sum
_{k\in\mathbb{Z},k\neq\pm n}\frac{|(q\Psi_{n,j},e^{i2\pi kx})|^{2}}%
{|\lambda_{n,j}-(2\pi k)^{2}|^{2}}\\
&  \qquad\leq\frac{1}{(4\pi^{2}(2n-1)-M)^{2}}\sum_{k\in\mathbb{Z},k\neq\pm
n}|(q\Psi_{n,j},e^{i2\pi kx})|^{2}<\frac{M^{2}}{(4\pi^{2}(2n-1)-M)^{2}}<1,
\end{align*}
which contradicts~\eqref{eq49} and completes the proof.
\end{proof}
\end{lemma}

Now, we consider the case $n=0$.

\begin{lemma}
\label{l5} If $M\leq4\pi^{2}/3$,\ then the statements

\textbf{(a)} $\lim_{m\rightarrow\infty}R_{m}(\lambda_{0})=0,$ \qquad
\textbf{(b)} $|(\Psi_{0},1)|>0$

are valid.

\begin{proof}
The proof is similar to the proof of Lemma~\ref{l1}.
\end{proof}
\end{lemma}

Before stating the main results, we write the following relations:
\begin{equation}
a_{1}((2\pi n)^{2})=\frac{-ab}{16\pi^{2}n^{2}}+\frac{(b^{2}-a^{2})\sin(4\pi
nc)}{64\pi^{3}n^{3}}+\frac{3(b-a)^{2}(\cos(4\pi nc)-1)}{128\pi^{4}n^{4}}
\label{eq50}%
\end{equation}
and
\begin{align}
&  q_{2n}+b_{1}((2\pi n)^{2})=q_{2n}+2Q_{0}Q_{2n}-S_{2n}\nonumber\\
&  \qquad=\frac{(a-b)(1-e^{-i4\pi nc})}{4\pi ni}+\frac{ab(1+e^{-i4\pi nc}%
)}{16\pi^{2}n^{2}}+\frac{(a^{2}-b^{2})(1-e^{-i4\pi nc})}{32\pi^{3}n^{3}i},
\label{eq51}%
\end{align}
which follow from~\eqref{eq18},~\eqref{eq39},~\eqref{eq33}
and~\eqref{eq9},~\eqref{eq34}-\eqref{eq36},~\eqref{eq33}, respectively, by
direct calculations.

Now we state the following interesting remark:

\begin{remark}
\label{r1} We have the formulas $q_{-k}=e^{i2\pi kc}q_{k}$ and $q_{-k}%
=\overline{q_{k}}$. These give, respectively,
\[
\arg q_{-k}=\arg q_{k}+2\pi kc\text{ and }\arg q_{-k}=-\arg q_{k}%
\]
and hence $\arg q_{k}=-\pi kc$, $q_{k}=e^{-i\pi kc}\left\vert q_{k}\right\vert
$, from which we obtain
\[
q_{2n}=e^{-i2\pi nc}\left\vert q_{2n}\right\vert
\]
and
\[
q_{n_{1}}q_{n_{2}}...q_{n_{k}}q_{2n-n_{1}-n_{2}-...-n_{k}}=\left\vert
q_{n_{1}}q_{n_{2}}...q_{n_{k}}q_{2n-n_{1}-n_{2}-...-n_{k}}\right\vert
e^{-i2\pi nc}.
\]
It means that $e^{i2\pi nc}\biggl(q_{2n}+\sum\limits_{k=1}^{\infty}%
b_{k}(\lambda)\biggr)$ is a real number.
\end{remark}

Using this remark and letting $m$ tend to infinity in equation~\eqref{eq4}, we
obtain the following main results. First, we consider the case $n\geq1$:

\begin{theorem}
\label{t7} \textbf{(a)} If $M\leq\dfrac{4\pi^{2}(2n-1)}{3}$, then
$\lambda_{n,j}$ is an eigenvalue of $L_{0}(q)$ if and only if it is either the
root of the equation
\begin{equation}
\lambda-(2\pi n)^{2}-\sum\limits_{k=1}^{\infty}a_{k}(\lambda)-e^{i2\pi
nc}\biggl(q_{2n}+\sum\limits_{k=1}^{\infty}b_{k}(\lambda)\biggr)=0
\label{eq52}%
\end{equation}
or the root of
\begin{equation}
\lambda-(2\pi n)^{2}-\sum\limits_{k=1}^{\infty}a_{k}(\lambda)+e^{i2\pi
nc}\biggl(q_{2n}+\sum\limits_{k=1}^{\infty}b_{k}(\lambda)\biggr)=0
\label{eq53}%
\end{equation}
in the set $D_{n}:=[(2\pi n)^{2}-M,(2\pi n)^{2}+M]$, where $n=1,2,\ldots$.
Moreover, the roots of~\eqref{eq52} and~\eqref{eq53} in $D_{n}$, coincide with
the $(2n)$th and $(2n+1)$st periodic eigenvalues $\lambda_{n,1}$ and
$\lambda_{n,2}$.
\end{theorem}

\begin{proof}
\textbf{(a)} By Lemma~\ref{l4}, letting $m$ tend to infinity in
equation~\eqref{eq4}, we obtain
\begin{equation}
\biggl(\lambda_{n,j}-(2\pi n)^{2}-\sum\limits_{k=1}^{\infty}a_{k}%
(\lambda_{n,j})\biggr)u_{n,j}=\biggl(q_{2n}+\sum\limits_{k=1}^{\infty}%
b_{k}(\lambda_{n,j})\biggr)\overline{u_{n,j}}, \label{eq54}%
\end{equation}
where $\biggl|\dfrac{u_{n,j}}{\overline{u_{n,j}}}\biggr|=1$. Therefore, we
have
\[
\biggl|\lambda_{n,j}-(2\pi n)^{2}-\sum\limits_{k=1}^{\infty}a_{k}%
(\lambda_{n,j})\biggr|=\biggl|q_{2n}+\sum\limits_{k=1}^{\infty}b_{k}%
(\lambda_{n,j})\biggr|.
\]
Now taking into account that
\[
\lambda_{n,j}-(2\pi n)^{2}-\sum\limits_{k=1}^{\infty}a_{k}(\lambda_{n,j})
\]
and
\[
e^{i2\pi nc}\biggl(q_{2n}+\sum\limits_{k=1}^{\infty}b_{k}(\lambda)\biggr)
\]
are real numbers (see Lemma~2.1.1~(b) of~\cite{17} and Remark~\ref{r1}), we
obtain that $\lambda_{n,j}$ is either the root of equation~\eqref{eq52} or the
root of~\eqref{eq53}.

Now, we prove that the roots of~\eqref{eq52} and~\eqref{eq53} lying in the
interval $D_{n}$ are the eigenvalues of $L_{0}$. To do this, we first give the
following relations:
\[
a_{1}(\lambda)=a_{1}((2\pi n)^{2})-\sum_{\substack{n_{1}=-\infty\\n_{1}%
\neq0,2n}}^{\infty}\frac{|q_{n_{1}}|^{2}[\lambda-(2\pi n)^{2}]}{4\pi^{2}%
n_{1}(2n-n_{1})[\lambda-(2\pi(n-n_{1}))^{2}]}%
\]
and
\[
b_{1}(\lambda)=b_{1}((2\pi n)^{2})-\sum_{\substack{n_{1}=-\infty\\n_{1}%
\neq0,2n}}^{\infty}\frac{q_{n_{1}}q_{2n-n_{1}}[\lambda-(2\pi n)^{2}]}{4\pi
^{2}n_{1}(2n-n_{1})[\lambda-(2\pi(n-n_{1}))^{2}]}.
\]

The equation ??(function)
\[
F_{n,j}(\lambda):=\lambda-(2\pi n)^{2}-a_{1}((2\pi n)^{2})+(-1)^{j}e^{i2\pi
nc}\bigl(q_{2n}+b_{1}((2\pi n)^{2})\bigr)
\]
has one root, for each $j$, in the interval $D_{n}=[(2\pi n)^{2}-M,(2\pi
n)^{2}+M]$ and
\[
|F_{n,j}(\lambda)|\geq|\lambda-(2\pi n)^{2}|-|a_{1}((2\pi n)^{2}%
)|-\bigl|q_{2n}+b_{1}((2\pi n)^{2})\bigr|.
\]
By Remark~\ref{r1} and~\eqref{eq51}, we write
\begin{align}
&  e^{i2\pi nc}\bigl(q_{2n}+b_{1}((2\pi n)^{2})\bigr)=e^{i2\pi nc}%
(q_{2n}+2Q_{0}Q_{2n}-S_{2n})\nonumber\\
&  \qquad=\frac{(a-b)\sin(2\pi nc)}{2\pi n}+\frac{ab\cos(2\pi nc)}{8\pi
^{2}n^{2}}+\frac{(a^{2}-b^{2})\sin(2\pi nc)}{16\pi^{3}n^{3}}. \label{eq55}%
\end{align}
Estimating $|a_{1}((2\pi n)^{2})|$\ and $|q_{2n}+b_{1}((2\pi n)^{2})|$
by~\eqref{eq50} and~\eqref{eq55}, we have\
\[
|F_{n,j}(\lambda)|>\frac{9\pi^{2}}{20},
\]
for all $\lambda$ from the boundary of $D_{n}$, for $n=1$. Now we define
\begin{equation}
G_{n,j}(\lambda):=\lambda-(2\pi n)^{2}-\sum\limits_{k=1}^{\infty}a_{k}%
(\lambda)+(-1)^{j}e^{i2\pi nc}\biggl(q_{2n}+\sum\limits_{k=1}^{\infty}%
b_{k}(\lambda)\biggr), \label{eq56}%
\end{equation}
for $j=1,2$. Estimating the summands of $|a_{1}(\lambda)-a_{1}((2\pi n)^{2}%
)|$, $|a_{k}(\lambda)|$, $|b_{1}(\lambda)-b_{1}((2\pi n)^{2})|$\ and
$|b_{k}(\lambda)|$ by considering the greatest summands of them,\ for $n=1$,
we obtain
\begin{align*}
|a_{1}(\lambda)-a_{1}((2\pi n)^{2})|  &  <1,\qquad|a_{k}(\lambda)|<\frac
{2}{\pi^{k-1}},\\
|b_{1}(\lambda)-b_{1}((2\pi n)^{2})|  &  <1,\qquad|b_{k}(\lambda)|<\frac
{2}{\pi^{k-1}},
\end{align*}
for $k\geq2$. Therefore, it follows by the geometric series formula that
\begin{align*}
|a_{1}(\lambda)-a_{1}((2\pi n)^{2})|+\sum\limits_{k=2}^{\infty}|a_{k}%
(\lambda)|  &  <\frac{\pi+1}{\pi-1}\\
|b_{1}(\lambda)-b_{1}((2\pi n)^{2})|+\sum\limits_{k=2}^{\infty}|b_{k}%
(\lambda)|  &  <\frac{\pi+1}{\pi-1},
\end{align*}
for $n=1$. Hence by~\eqref{eq56}, we obtain
\begin{align*}
&  |G_{n,j}-F_{n,j}|=\biggl|a_{1}(\lambda)-a_{1}((2\pi n)^{2})+\sum
\limits_{k=2}^{\infty}a_{k}(\lambda)+(-1)^{j}e^{i2\pi nc}\biggl(b_{1}%
(\lambda)-b_{1}((2\pi n)^{2})+\sum\limits_{k=2}^{\infty}b_{k}(\lambda
)\biggr)\biggr|\\
&  \qquad\leq|a_{1}(\lambda)-a_{1}((2\pi n)^{2})|+\sum\limits_{k=2}^{\infty
}|a_{k}(\lambda)|+|b_{1}(\lambda)-b_{1}((2\pi n)^{2})|+\sum\limits_{k=2}%
^{\infty}|b_{k}(\lambda)|\\
&  \qquad<\frac{2(\pi+1)}{\pi-1},
\end{align*}
for all $\lambda$ from the boundary of $D_{n}$, for $n=1$. Similar estimations
can be obtained also for $n\geq2$. Therefore $|G_{n,j}(\lambda)-F_{n,j}%
(\lambda)|<|F_{n,j}(\lambda)|$ holds for all $\lambda$ from the boundary of
$D_{n}$ and by Rouche's theorem, $G_{n,j}(\lambda)$ has one root in the set
$D_{n}$, for $j=1$ and $j=2$. Hence, $L_{0}$ has one eigenvalue (counting
multiplicity) in $D_{n}$, which is the root of~\eqref{eq52} and it has one
eigenvalue (counting multiplicity) in $D_{n}$, which is the root
of~\eqref{eq53}. On the other hand, each of the equations~\eqref{eq52}
and~\eqref{eq53} has exactly one root (counting multiplicity) in $D_{n}$.
Thus, $\lambda\in D_{n}$ is an eigenvalue of $L_{0}$ if and only if, it is
either the root of~\eqref{eq52} or the root of~\eqref{eq53} and the roots
of~\eqref{eq52} and~\eqref{eq53} coincide with the eigenvalues $\lambda_{n,1}%
$\ and $\lambda_{n,2}$\ of $L_{0}$.
\end{proof}

Now, we consider the case $n=0$.

\begin{theorem}
\label{t8} If $M\leq4\pi^{2}/3$, then the first periodic eigenvalue
$\lambda_{0}$ is the root of the equation
\begin{equation}
\lambda-\sum\limits_{k=1}^{\infty}a_{k}(\lambda)=0, \label{eq57}%
\end{equation}
in the set $D_{0}=[-M,M]$. Moreover,~\eqref{eq57} has exactly one\ root
(counting multiplicity) in the set $D_{0}$ and this root coincides with the
first eigenvalue $\lambda_{0}$ of $L_{0}$.

\begin{proof}
\textbf{(a)} Iterating the equation $\lambda_{N}(\Psi_{N},1)=(q\Psi_{N},1)$
$m$ times, for $N=0$, by isolating the terms containing $(\Psi_{0},1)$ gives
\[
\biggl(\lambda_{0}-\sum\limits_{k=1}^{m}a_{k}(\lambda_{0})\biggr)(\Psi
_{0},1)=R_{m}(\lambda_{0}),
\]
where $a_{k}$ and $R_{m}$\ are defined by~\eqref{eq5}and~\eqref{eq7}.\ Letting
$m$ tend to infinity in the last equation, by Lemma~\ref{l5}, we
obtain~\eqref{eq57}. Let $F_{0}(\lambda):=\lambda-a_{1}(0)$ and
\[
G_{0}(\lambda):=\lambda-\sum\limits_{k=1}^{\infty}a_{1}(\lambda).
\]
Then, using the relations $Q_{0}=ac/2$,
\begin{align*}
a_{1}(0)  &  =\sum_{\substack{k=-\infty\\k\neq0}}^{\infty}\frac{|q_{k}|^{2}%
}{-(2\pi(-k))^{2}}=-\sum_{\substack{k=-\infty\\k\neq0}}^{\infty}\frac
{|q_{k}|^{2}}{4\pi^{2}k^{2}}=Q_{0}^{2}-\sum_{\substack{k=-\infty\\k\neq
0}}^{\infty}Q_{k}Q_{-k}\\
&  =Q_{0}^{2}-\int_{0}^{1}\biggl(\sum_{k=-\infty}^{\infty}Q_{k}e^{i2\pi
kx}\biggr)^{2}dx=Q_{0}^{2}-\int_{0}^{1}Q^{2}(x)dx=Q_{0}^{2}-S_{0},
\end{align*}
and
\[
S_{0}=\int_{0}^{1}S(x)dx=\int_{0}^{1}Q^{2}(x)dx=\int_{0}^{c}a^{2}x^{2}%
dx+\int_{c}^{1}(bx-b)^{2}dx=\frac{a^{2}c^{2}}{3},
\]
by direct calculation we obtain
\[
-a_{1}(0)=S_{0}-Q_{0}^{2}=\frac{a^{2}c^{2}}{3}-\frac{a^{2}c^{2}}{4}%
=\frac{a^{2}c^{2}}{12}.
\]
If $c\in(0,1/2)$ and $M<4\pi^{2}/3$, then $|ac|\leq(1/2)M<2\pi^{2}/3$ since in
this case $b<|a|$ and $|a|=M$. If $c\in(1/2,1)$, then $|ac|=(1-c)b\leq
(1-1/2)M<2\pi^{2}/3$. In this case $b>|a|$ and $b=M$. Therefore, in any case
$|ac|<2\pi^{2}/3$ if $M<4\pi^{2}/3$ and we have $A_{1,0}(0)<\pi^{4}/27$. It is
easy to verify that the inequality $|G_{0}(\lambda)-F_{0}(\lambda
)|<|F_{0}(\lambda)|$ holds for all $\lambda$ from the boundary of $D_{0}$ and
by Rouche's theorem, $G_{0}(\lambda)$ has one root in the set $D_{0}$. Hence,
$L_{0}$ has one eigenvalue (counting multiplicity) in $D_{0}$, which is the
root of~\eqref{eq57}. On the other hand, equation~\eqref{eq57} has exactly one
root (counting multiplicity) in $D_{0}$. Thus, $\lambda\in D_{0}$ is an
eigenvalue of $L_{0}$ if and only if, it is the root of~\eqref{eq57} and the
root of~\eqref{eq57}\ in $D_{0}$ coincides with the first eigenvalue
$\lambda_{0}$\ of $L_{0}$.
\end{proof}
\end{theorem}

Now, we can use numerical methods by taking finite sums instead of the
infinite series in~\eqref{eq52},~\eqref{eq53} and~\eqref{eq57}. In this case,
we write
\[
\lambda-(2\pi n)^{2}-\sum\limits_{k=1}^{r}a_{s,k,n}(\lambda)+(-1)^{j}e^{i2\pi
nc}\biggl(q_{2n}+\sum\limits_{k=1}^{r}b_{s,k,n}(\lambda)\biggr)=0,
\]
for $j=1$ and $j=2$, and\
\[
\lambda-\sum\limits_{k=1}^{r}a_{s,k,0}(\lambda)=0,
\]
respectively, where
\begin{align*}
a_{s,k,n}(\lambda)  &  =\sum_{n_{1},n_{2},...,n_{k}=-s}^{s}\frac{q_{n_{1}%
}q_{n_{2}}\cdots q_{n_{k}}q_{-n_{1}-n_{2}-\cdots-n_{k}}}{[\lambda
-(2\pi(n-n_{1}))^{2}]\cdots\lbrack\lambda-(2\pi(n-n_{1}-\cdots-n_{k}))^{2}%
]},\\
b_{s,k,n}(\lambda)  &  =\sum_{n_{1},n_{2},...,n_{k}=-s}^{s}\frac{q_{n_{1}%
}q_{n_{2}}\cdots q_{n_{k}}q_{2n-n_{1}-n_{2}-\cdots-n_{k}}}{[\lambda
-(2\pi(n-n_{1}))^{2}]\cdots\lbrack\lambda-(2\pi(n-n_{1}-\cdots-n_{k}))^{2}]}.
\end{align*}
We define the functions
\begin{equation}
K_{n,j}(\lambda):=\lambda-(2\pi n)^{2}-g_{n,j}(\lambda) \label{eq58}%
\end{equation}
and
\begin{equation}
K_{0}(\lambda):=\lambda-g_{0}(\lambda), \label{eq59}%
\end{equation}
where
\begin{equation}
g_{n,j}(\lambda)=\sum\limits_{k=1}^{r}a_{s,k,n}(\lambda)-(-1)^{j}e^{i2\pi
nc}\biggl(q_{2n}+\sum\limits_{k=1}^{r}b_{s,k,n}(\lambda)\biggr) \label{eq60}%
\end{equation}
and
\begin{equation}
g_{0}(\lambda)=\sum\limits_{k=1}^{r}a_{s,k,0}(\lambda). \label{eq61}%
\end{equation}
Then,
\begin{equation}
\lambda=(2\pi n)^{2}+g_{n,j}(\lambda), \label{eq62}%
\end{equation}
for $j=1$ and $j=2$, and $n\geq1$.

Now we state another main result.

\begin{theorem}
\label{t9} If $M\leq\dfrac{4\pi^{2}(2n-1)}{3}$,\emph{ }then for all $x$ and
$y$ from the interval $D_{n}=[(2\pi n)^{2}-M,(2\pi n)^{2}+M]$, the relations
\begin{gather}
|g_{n,j}(x)-g_{n,j}(y)|\leq C_{n}|x-y|,\label{eq63}\\
C_{n}=\frac{4(b-a)^{2}}{\pi(4\pi^{2}(2n-1)-M)[\pi(4\pi^{2}(2n-1)-M)-(b-a)]}%
\leq\frac{4}{\pi(\pi-1)}<1,\nonumber
\end{gather}
hold for $j=1,2$, and equation~\eqref{eq62} has a unique solution $\rho_{n,j}$
in $D_{n}$, for each $j$, where $n=1,2,\ldots$. Moreover
\begin{align}
&  |\lambda_{n,j}-\rho_{n,j}|<\frac{3(b-a)^{r+2}}{2\pi^{r+1}(4\pi
^{2}(2n-1)-M)^{r}[\pi(4\pi^{2}(2n-1)-M)-(b-a)](1-C_{n})}\nonumber\\
&  \qquad+\frac{6(b-a)^{2}}{\pi^{2}(s+1)^{2}[4\pi^{2}(s+1)|s+1-2n|-M](1-C_{n}%
)}, \label{eq64}%
\end{align}
for $j=1,2$.
\end{theorem}

\begin{proof}
First we prove~\eqref{eq63} by using the mean-value theorem. To do this, we
estimate $|g_{n,j}^{\prime}(\lambda)|=|\dfrac{d}{d\lambda}g_{n,j}(\lambda)|$.
By~\eqref{eq60} we have
\begin{align*}
|g_{n,j}^{\prime}(\lambda)|  &  =\biggl|\sum\limits_{k=1}^{r}\frac{d}%
{d\lambda}a_{s,k,n}(\lambda)+(-1)^{j}e^{i2\pi nc}\sum\limits_{k=1}^{r}\frac
{d}{d\lambda}b_{s,k,n}(\lambda)\biggr|\\
&  \leq\sum\limits_{k=1}^{r}\bigl|\frac{d}{d\lambda}a_{s,k,n}(\lambda
)\bigr|+\sum\limits_{k=1}^{r}\bigl|\frac{d}{d\lambda}b_{s,k,n}(\lambda)\bigr|.
\end{align*}
First let us estimate the summands of the first term $|\frac{d}{d\lambda
}(A_{s,k,n}(\lambda))|$:
\begin{align*}
&  \biggl|\frac{d}{d\lambda}(a_{s,k,n}(\lambda))\biggr|=\biggl|\sum
_{n_{1},n_{2},...,n_{k}=-s}^{s}\frac{d}{d\lambda}\frac{q_{n_{1}}q_{n_{2}%
}\cdots q_{n_{k}}q_{-n_{1}-n_{2}-\cdots-n_{k}}}{[\lambda-(2\pi(n-n_{1}%
))^{2}]\cdots\lbrack\lambda-(2\pi(n-n_{1}-\cdots-n_{k}))^{2}]}\biggr|\\
&  \qquad<\frac{2(b-a)^{k+1}}{\pi^{k+1}(4\pi^{2}(2n-1)-M)^{k+1}}\leq\frac
{2}{\pi^{k+1}},
\end{align*}
for $k\geq1$. Thus, by the geometric series formula, we obtain
\[
\sum\limits_{k=1}^{r}\bigl|\frac{d}{d\lambda}a_{s,k,n}(\lambda)\bigr|<\frac
{2}{\pi(\pi-1)}.
\]
In the same way, one can prove that
\[
\sum\limits_{k=1}^{r}\bigl|\frac{d}{d\lambda}b_{s,k,n}(\lambda)\bigr|<\frac
{2}{\pi(\pi-1)}.
\]
Hence,
\[
|g_{n,j}^{\prime}(\lambda)|<\frac{4(b-a)^{2}}{\pi(4\pi^{2}(2n-1)-M)[\pi
(4\pi^{2}(2n-1)-M)-(b-a)]}=C_{n}\leq\frac{4}{\pi(\pi-1)}<1.
\]
Since the inequalities
\begin{equation}
|g_{j}^{\prime}(\lambda)|<C_{n}<1 \label{eq65}%
\end{equation}
hold for all $x,y\in D_{n}$,~\eqref{eq63} holds by the mean value theorem and
equation~\eqref{eq62} has a unique solution $\rho_{n,j}$ in $D_{n}$, for each
$j$ ($j=1,2$), by the contraction mapping theorem.

Now let us prove~\eqref{eq64}. By~\eqref{eq58}, we have $K_{n,j}(x)=x-(2\pi
n)^{2}-g_{n,j}(x)$ and by the definition of $\rho_{n,j}$, we write
$K_{n,j}(\rho_{n,j})=0$, for $j=1,2$. Therefore by~\eqref{eq52},~\eqref{eq53}
and~\eqref{eq60}, we obtain
\begin{align}
&  |K_{n,j}(\lambda_{n,j})-K_{n,j}(\rho_{n,j})|=|K_{n,j}(\lambda
_{n,j})|\nonumber\\
&  \qquad=\biggl|\lambda_{n,j}-(2\pi n)^{2}-\sum\limits_{k=1}^{r}%
a_{s,k,n}(\lambda_{n,j})-(-1)^{j}e^{i2\pi nc}\biggl(q_{2n}+\sum\limits_{k=1}%
^{r}b_{s,k,n}(\lambda_{n,j})\biggr)\biggr|\nonumber\\
&  \qquad\leq\biggl|\sum\limits_{k=1}^{\infty}a_{k,n}(\lambda_{n,j}%
)-\sum\limits_{k=1}^{r}a_{s,k,n}(\lambda_{n,j})\biggr|+\biggl|\sum
\limits_{k=1}^{\infty}b_{k,n}(\lambda_{n,j})-\sum\limits_{k=1}^{r}%
b_{s,k,n}(\lambda_{n,j})\biggr|. \label{eq66}%
\end{align}
For the estimation of the first term of the right-hand side of~\eqref{eq66},
we obtain
\begin{align*}
&  \biggl|\sum\limits_{k=1}^{\infty}a_{k,n}(\lambda_{n,j})-\sum\limits_{k=1}%
^{r}a_{s,k,n}(\lambda_{n,j})\biggr|\\
&  \qquad\leq\biggl|\sum\limits_{k=1}^{\infty}a_{k,n}(\lambda_{n,j}%
)-\sum\limits_{k=1}^{r}a_{k,n}(\lambda_{n,j})\biggr|+\biggl|\sum
\limits_{k=1}^{r}a_{k,n}(\lambda_{n,j})-\sum\limits_{k=1}^{r}a_{s,k,n}%
(\lambda_{n,j})\biggr|\\
&  \qquad\leq\sum\limits_{k=r+1}^{\infty}|a_{k,n}(\lambda_{n,j})|+\sum
\limits_{k=1}^{r}|a_{k,n}(\lambda_{n,j})-a_{s,k,n}(\lambda_{n,j})|.
\end{align*}
Using the estimations for $\sum\limits_{k=r+1}^{\infty}|a_{k,n}(\lambda
_{n,j})|$\ and $\sum\limits_{k=1}^{r}|a_{k,n}(\lambda_{n,j})-a_{s,k,n}%
(\lambda_{n,j})|$, by considering the greatest summands of them, we obtain
\begin{align*}
&  \biggl|\sum\limits_{k=1}^{\infty}a_{k,n}(\lambda_{n,j})-\sum\limits_{k=1}%
^{r}a_{s,k,n}(\lambda_{n,j})\biggr|<\sum\limits_{k=r+1}^{\infty}%
\frac{3(b-a)^{k+1}}{4\pi^{k+1}(4\pi^{2}(2n-1)-M)^{k}}\\
&  \qquad+\frac{3(b-a)^{2}}{\pi^{2}(s+1)^{2}[4\pi^{2}(s+1)|s+1-2n|-M]}%
\end{align*}
and
\begin{align}
&  \biggl|\sum\limits_{k=1}^{\infty}a_{k,n}(\lambda_{n,j})-\sum\limits_{k=1}%
^{r}a_{s,k,n}(\lambda_{n,j})\biggr|<\frac{3(b-a)^{r+2}}{4\pi^{r+1}(4\pi
^{2}(2n-1)-M)^{r}[\pi(4\pi^{2}(2n-1)-M)-(b-a)]}\nonumber\\
&  \qquad+\frac{3(b-a)^{2}}{\pi^{2}(s+1)^{2}[4\pi^{2}(s+1)|s+1-2n|-M]}.
\label{eq67}%
\end{align}
In the same way, one can obtain
\begin{align}
&  \biggl|\sum\limits_{k=1}^{\infty}b_{k,n}(\lambda_{n,j})-\sum\limits_{k=1}%
^{r}b_{s,k,n}(\lambda_{n,j})\biggr|<\frac{3(b-a)^{r+2}}{4\pi^{r+1}(4\pi
^{2}(2n-1)-M)^{r}[\pi(4\pi^{2}(2n-1)-M)-(b-a)]}\nonumber\\
&  \qquad+\frac{3(b-a)^{2}}{\pi^{2}(s+1)^{2}[4\pi^{2}(s+1)|s+1-2n|-M]}.
\label{eq68}%
\end{align}
Thus, by~\eqref{eq66}--\eqref{eq68} we have
\begin{align}
&  |K_{n,j}(\lambda_{n,j})-K_{n,j}(\rho_{n,j})|<\frac{3(b-a)^{r+2}}{2\pi
^{r+1}(4\pi^{2}(2n-1)-M)^{r}[\pi(4\pi^{2}(2n-1)-M)-(b-a)]}\nonumber\\
&  \qquad+\frac{6(b-a)^{2}}{\pi^{2}(s+1)^{2}[4\pi^{2}(s+1)|s+1-2n|-M]}.
\label{eq69}%
\end{align}
To apply the mean value theorem, we estimate $|K_{n,j}^{\prime}(\lambda)|$:
\begin{equation}
|K_{n,j}^{\prime}(\lambda)|=|1-g_{n,j}^{\prime}(\lambda)|\geq|1-|g_{n,j}%
^{\prime}(\lambda)||\geq1-C_{n}. \label{eq70}%
\end{equation}
By the mean value formula,~\eqref{eq69} and~\eqref{eq70}, we obtain
\[
|K_{n,j}(\lambda_{n,j})-K_{n,j}(\rho_{n,j})|=|K_{n,j}^{\prime}(\xi
)||\lambda_{n,j}-\rho_{n,j}|,\qquad\xi\in\lbrack(2\pi n)^{2}-M,(2\pi
n)^{2}+M]
\]
and
\begin{align*}
&  |\lambda_{n,j}-\rho_{n,j}|=\frac{|K_{n,j}(\lambda_{n,j})-K_{n,j}(\rho
_{n,j})|}{|K_{n,j}^{\prime}(\xi)|}\\
&  \qquad<\frac{3(b-a)^{r+2}}{2\pi^{r+1}(4\pi^{2}(2n-1)-M)^{r}[\pi(4\pi
^{2}(2n-1)-M)-(b-a)](1-C_{n})}\\
&  \qquad+\frac{6(b-a)^{2}}{\pi^{2}(s+1)^{2}[4\pi^{2}(s+1)|s+1-2n|-M](1-C_{n}%
)}.
\end{align*}

\end{proof}

Now, we give an analogous theorem to Theorem~\ref{t9} for the case $n=0$.

\begin{theorem}
\label{t10} If $M\leq4\pi^{2}/3$, then for all $x$ and $y$ from the interval
$D_{0}=[-M,M]$ the relations
\begin{gather*}
|g_{0}(x)-g_{0}(y)|\leq C_{0}|x-y|,\\
C_{0}=\frac{3(b-a)^{2}}{4\pi(2\pi^{2}-M)[\pi(4\pi^{2}-M)-(b-a)]}\leq\frac
{3}{\pi(\pi-1)}<1,
\end{gather*}
hold and the equation\
\[
\lambda=g_{0}(\lambda)
\]
has a unique solution $\rho_{0}$ in $D_{0}$, where $g_{0}(\lambda)$\ is
defined by~\eqref{eq61}. Moreover
\begin{align}
&  |\lambda_{0}-\rho_{0}|<\frac{9(b-a)^{r+2}}{16\pi^{r+1}(4\pi^{2}%
-M)^{r-1}(2\pi^{2}-M)[\pi(4\pi^{2}-M)-(b-a)](1-C_{0})}\nonumber\\
&  \qquad+\frac{3(b-a)^{2}}{\pi^{2}(s+1)^{2}[4\pi^{2}(s+1)^{2}-M](1-C_{0})}.
\label{eq71}%
\end{align}

\end{theorem}

\begin{proof}
The proof is similar to the proof of Theorem~\ref{t9} by the following
estimates:
\begin{align*}
&  \biggl|\frac{d}{d\lambda}(a_{s,k,0}(\lambda))\biggr|=\biggl|\sum
_{n_{1},n_{2},...,n_{k}=-s}^{s}\frac{d}{d\lambda}\frac{q_{n_{1}}q_{n_{2}%
}\cdots q_{n_{k}}q_{-n_{1}-n_{2}-\cdots-n_{k}}}{[\lambda-(2\pi(-n_{1}%
))^{2}]\cdots\lbrack\lambda-(2\pi(-n_{1}-\cdots-n_{k}))^{2}]}\biggr|\\
&  \qquad<\frac{3(b-a)^{k+1}}{4\pi^{k+1}(4\pi^{2}-M)^{k}(2\pi^{2}-M)}\leq
\frac{3}{\pi^{k+1}},
\end{align*}
for $k\geq1$,
\[
|g_{0}^{\prime}(\lambda)|=\biggl|\sum\limits_{k=1}^{r}\frac{d}{d\lambda
}a_{s,k,0}(\lambda)\biggr|<\frac{3(b-a)^{2}}{4\pi(2\pi^{2}-M)[\pi(4\pi
^{2}-M)-(b-a)]}=C_{0}\leq\frac{3}{\pi(\pi-1)}<1,
\]%
\[
|a_{k,0}(\lambda_{0})|<\dfrac{9(b-a)^{k+1}}{16\pi^{k+1}(4\pi^{2}-M)^{k-1}%
(2\pi^{2}-M)}\leq\dfrac{6}{\pi^{k-1}},
\]
for $k\geq1$, and
\begin{align*}
&  |K_{0}(\lambda_{0})-K_{0}(\rho_{0})|\leq\sum\limits_{k=r+1}^{\infty
}|a_{k,0}(\lambda_{0})|+\sum\limits_{k=1}^{r}|a_{k,0}(\lambda_{0}%
)-\sum\limits_{k=1}^{r}a_{s,k,0}(\lambda_{0})|\\
&  \qquad<\frac{9(b-a)^{r+2}}{16\pi^{r+1}(4\pi^{2}-M)^{r-1}(2\pi^{2}%
-M)[\pi(4\pi^{2}-M)-(b-a)]}+\frac{3(b-a)^{2}}{\pi^{2}(s+1)^{2}[4\pi
^{2}(s+1)^{2}-M]},
\end{align*}
where $K_{0}(\lambda)$\ is defined by~\eqref{eq59}.
\end{proof}

Now let us approximate $\rho_{n,j}$ by the fixed point iterations:
\begin{equation}
x_{n,i+1}=(2\pi n)^{2}+g_{n,1}(x_{n,i}), \label{eq72}%
\end{equation}
and
\begin{equation}
y_{n,i+1}=(2\pi n)^{2}+g_{n,2}(y_{n,i}), \label{eq73}%
\end{equation}
where $g_{n,j}(x)$ ($j=1,2$) is defined by~\eqref{eq60}. Since
\begin{align*}
|g_{n,j}(\lambda_{n,j})|  &  =\biggl|\sum\limits_{k=1}^{r}a_{s,k,n}%
(\lambda_{n,j})-(-1)^{j}e^{i2\pi nc}\biggl(q_{2n}+\sum\limits_{k=1}%
^{r}b_{s,k,n}(\lambda_{n,j})\biggr)\biggr|\\
&  \leq\biggl|\sum\limits_{k=1}^{r}a_{s,k,n}(\lambda_{n,j}%
)\biggr|+\biggl|q_{2n}+\sum\limits_{k=1}^{r}b_{s,k,n}(\lambda_{n,j})\biggr|\\
&  \leq\sum\limits_{k=1}^{r}|a_{s,k,n}(\lambda_{n,j})|+|q_{2n}|+\sum
\limits_{k=1}^{r}|b_{s,k,n}(\lambda_{n,j})|,
\end{align*}
first let us estimate $|a_{s,k,n}(\lambda)|$. The estimation of $|b_{s,k,n}%
(\lambda)|$ is similar.
\begin{equation}
\sum\limits_{k=1}^{r}|a_{s,k,n}(\lambda)|<\sum\limits_{k=1}^{r}\frac
{3(b-a)^{k+1}}{4\pi^{k+1}(4\pi^{2}(2n-1)-M)^{k}}\leq\sum\limits_{k=1}^{r}%
\frac{2}{\pi^{k-1}}<\frac{2\pi}{\pi-1}. \label{eq74}%
\end{equation}
Similarly, we obtain
\[
|q_{2n}|+\sum\limits_{k=1}^{r}|b_{s,k,n}(\lambda)|<\frac{(b-a)}{2\pi n}%
+\frac{2\pi}{\pi-1}.
\]
Hence, we have
\begin{equation}
|g_{n,j}(\lambda_{n,j})|<\frac{(b-a)}{2\pi n}+\sum\limits_{k=1}^{r}%
\frac{3(b-a)^{k+1}}{2\pi^{k+1}(4\pi^{2}(2n-1)-M)^{k}}<\frac{(b-a)}{2\pi
n}+\frac{4\pi}{\pi-1}. \label{eq75}%
\end{equation}
On the other hand, writing $4\pi^{2}(2n-1)$ instead of $4\pi^{2}(2n-1)-M$
in~\eqref{eq74} and~\eqref{eq75}, we obtain
\begin{align}
|g_{n,j}((2\pi n)^{2})|  &  \leq\sum\limits_{k=1}^{r}|a_{s,k,n}((2\pi
n)^{2})|+|q_{2n}|+\sum\limits_{k=1}^{r}|b_{s,k,n}((2\pi n)^{2})|\nonumber\\
&  <\frac{(b-a)}{2\pi n}+\sum\limits_{k=1}^{r}\frac{3(b-a)^{k+1}}{2\pi
^{k+1}(4\pi^{2}(2n-1))^{k}}<\frac{(b-a)}{2\pi n}+\frac{3(b-a)^{2}}{2\pi
\lbrack4\pi^{3}(2n-1)-(b-a)]}, \label{eq76}%
\end{align}
since $|(2\pi n)^{2}-(2\pi k)^{2}|\geq4\pi^{2}(2n-1)$, for $n=1,2,\ldots$. Now
we state the following result.

\begin{theorem}
\label{t11} If $M\leq\dfrac{4\pi^{2}(2n-1)}{3}$, then the following
estimations hold for the sequences $\{x_{n,i}\}$ and $\{y_{n,i}\}$ defined
by~\eqref{eq72} and~\eqref{eq73}:
\begin{align}
|x_{n,i}-\rho_{n,1}|  &  <(C_{n})^{i}\biggl(\frac{(b-a)}{2\pi n(1-C_{n}%
)}+\frac{3(b-a)^{2}}{2\pi\lbrack4\pi^{3}(2n-1)-(b-a)](1-C_{n})}%
\biggr),\label{eq77}\\
|y_{n,i}-\rho_{n,2}|  &  <(C_{n})^{i}\biggl(\frac{(b-a)}{2\pi n(1-C_{n}%
)}+\frac{3(b-a)^{2}}{2\pi\lbrack4\pi^{3}(2n-1)-(b-a)](1-C_{n})}\biggr),
\label{eq78}%
\end{align}
for $i=1,2,3,\ldots$, where $C_{n}$ is defined in~\eqref{eq63} and
$n=1,2,\ldots$.
\end{theorem}

\begin{proof}
Without loss of generality, we can take $x_{n,0}=(2\pi n)^{2}$.
By~\eqref{eq62},~\eqref{eq63}, and~\eqref{eq72} we have
\begin{align*}
&  |x_{n,i}-\rho_{n,1}|=|(2\pi n)^{2}+g_{n,1}(x_{n,i-1})-((2\pi n)^{2}%
+g_{n,1}(\rho_{n,1}))|\\
&  \qquad=|g_{n,1}(x_{n,i-1})-g_{n,1}(\rho_{n,1})|<C_{n}|x_{n,i-1}-\rho
_{n,1}|<(C_{n})^{i}|x_{n,0}-\rho_{n,1}|.
\end{align*}
Therefore it is enough to estimate $|x_{n,0}-\rho_{n,1}|$. By definitions of
$\rho_{n,j}$ and $x_{n,0}$ we obtain
\[
\rho_{n,1}-x_{n,0}=g_{n,1}(\rho_{n,1})+(2\pi n)^{2}-x_{n,0}=g_{n,1}(\rho
_{n,1})-g_{n,1}(x_{n,0})+g_{n,1}((2\pi n)^{2})
\]
and by the mean value theorem there exists $x\in\lbrack(2\pi n)^{2}-M,(2\pi
n)^{2}+M]$ such that
\[
g_{n,1}(\rho_{n,1})-g_{n,1}(x_{n,0})=g_{n,1}^{\prime}(x)(\rho_{n,1}-x_{n,0}).
\]
The last two equalities imply that
\[
(\rho_{n,1}-x_{n,0})(1-g_{n,1}^{\prime}(x))=g_{n,1}((2\pi n)^{2}).
\]
Hence by~\eqref{eq65} and~\eqref{eq76} we have
\[
|\rho_{n,1}-x_{n,0}|\leq\frac{|g_{n,1}((2\pi n)^{2})|}{1-C_{n}}<\frac
{(b-a)}{2\pi n(1-C_{n})}+\frac{3(b-a)^{2}}{2\pi\lbrack4\pi^{3}%
(2n-1)-(b-a)](1-C_{n})}%
\]
and
\[
|x_{n,i}-\rho_{n,1}|<(C_{n})^{i}\biggl(\frac{(b-a)}{2\pi n(1-C_{n})}%
+\frac{3(b-a)^{2}}{2\pi\lbrack4\pi^{3}(2n-1)-(b-a)](1-C_{n})}\biggr).
\]
One can easily show by the same way that
\[
|y_{n,i}-\rho_{n,2}|<(C_{n})^{i}\biggl(\frac{(b-a)}{2\pi n(1-C_{n})}%
+\frac{3(b-a)^{2}}{2\pi\lbrack4\pi^{3}(2n-1)-(b-a)](1-C_{n})}\biggr)
\]
holds.
\end{proof}

An analogous theorem to Theorem~\ref{t11} can be stated for the case $n=0$.

\begin{theorem}
\label{t12} If $M\leq4\pi^{2}/3$, then the following estimation holds for the
sequence $\{x_{0,i}\}$ defined by $x_{0,i}=g_{0}(x_{0,i})$, where
$g_{0}(\lambda)$ is defined by~\eqref{eq61}:
\begin{equation}
|x_{0,i}-\rho_{0}|\leq(C_{0})^{i}\frac{(b-a)^{2}}{2\pi\lbrack2\pi
^{3}-(b-a)](1-C_{0})}, \label{eq79}%
\end{equation}
where $C_{0}$ is defined in~Theorem~\ref{t10}.
\end{theorem}

\begin{proof}
The proof is similar to the proof of Theorem~\ref{t11} by the following
estimates:
\[
|g_{0}(0)|\leq\sum\limits_{k=1}^{r}|A_{s,k,0}(0)|<\sum\limits_{k=1}^{r}%
\frac{(b-a)^{k+1}}{2^{k+1}\pi^{3k+1}}<\frac{(b-a)^{2}}{2\pi\lbrack2\pi
^{3}-(b-a)]}.
\]

\end{proof}

Thus by~\eqref{eq64},~\eqref{eq71},~\eqref{eq77}-\eqref{eq79}, we have the
approximations $x_{0,i}$, $x_{n,i}$, and $y_{n,i}$ for $\lambda_{0}$,
$\lambda_{n,1}$, and $\lambda_{n,2}$, respectively, with the errors
\begin{align*}
&  |\lambda_{0}-x_{0,i}|<\frac{9(b-a)^{r+2}}{16\pi^{r+1}(4\pi^{2}%
-M)^{r-1}(2\pi^{2}-M)[\pi(4\pi^{2}-M)-(b-a)](1-C_{0})}\\
&  \qquad+\frac{3(b-a)^{2}}{\pi^{2}(s+1)^{2}[4\pi^{2}(s+1)^{2}-M](1-C_{0}%
)}+(C_{0})^{i}\frac{(b-a)^{2}}{2\pi\lbrack2\pi^{3}-(b-a)](1-C_{0})},
\end{align*}%
\begin{align*}
&  |\lambda_{n,1}-x_{n,i}|<\frac{3(b-a)^{r+2}}{2\pi^{r+1}(4\pi^{2}%
(2n-1)-M)^{r}[\pi(4\pi^{2}(2n-1)-M)-(b-a)](1-C_{n})}\\
&  \qquad+\frac{6(b-a)^{2}}{\pi^{2}(s+1)^{2}[4\pi^{2}(s+1)|s+1-2n|-M](1-C_{n}%
)}\\
&  \qquad+(C_{n})^{i}\biggl(\frac{(b-a)}{2\pi n(1-C_{n})}+\frac{3(b-a)^{2}%
}{2\pi\lbrack4\pi^{3}(2n-1)-(b-a)](1-C_{n})}\biggr),
\end{align*}
and
\begin{align*}
&  |\lambda_{n,2}-y_{n,i}|<\frac{3(b-a)^{r+2}}{2\pi^{r+1}(4\pi^{2}%
(2n-1)-M)^{r}[\pi(4\pi^{2}(2n-1)-M)-(b-a)](1-C_{n})}\\
&  \qquad+\frac{6(b-a)^{2}}{\pi^{2}(s+1)^{2}[4\pi^{2}(s+1)|s+1-2n|-M](1-C_{n}%
)}\\
&  \qquad+(C_{n})^{i}\biggl(\frac{(b-a)}{2\pi n(1-C_{n})}+\frac{3(b-a)^{2}%
}{2\pi\lbrack4\pi^{3}(2n-1)-(b-a)](1-C_{n})}\biggr).
\end{align*}
By these error formulas it is clear that the error gets smaller as $r$ and $s$\ increase.

Now, we consider the operator $L_{\pi}(q)$ which is associated with the
antiperiodic boundary conditions. The analogous formula to~\eqref{eq4} is
\begin{align}
&  \biggl(\mu_{n,j}-(2n-1)^{2}\pi^{2}-\sum\limits_{k=1}^{m}\eta_{k}(\mu
_{n,j})\biggr)(\Phi_{n,j},e^{i(2n-1)\pi x})\nonumber\\
&  \qquad-\biggl(q_{2n-1}+\sum\limits_{k=1}^{m}\nu_{k}(\mu_{n,j}%
)\biggr)(\Phi_{n,j},e^{-i(2n-1)\pi x})=\delta_{m}(\mu_{n,j}), \label{eq80}%
\end{align}
where $\Phi_{n,j}$\ is an eigenfunction corresponding to the antiperiodic
eigenvalue $\mu_{n,j}$ and
\begin{align*}
\eta_{k}(\mu)  &  =\sum_{n_{1},n_{2},...,n_{k}}\frac{q_{n_{1}}q_{n_{2}}\cdots
q_{n_{k}}q_{-n_{1}-n_{2}-\cdots-n_{k}}}{[\mu-(2(n-n_{1})-1)^{2}\pi^{2}%
]\cdots\lbrack\mu-(2(n-n_{1}-\cdots-n_{k})-1)^{2}\pi^{2}]},\\
\nu_{k}(\mu)  &  =\sum_{n_{1},n_{2},...,n_{k}}\frac{q_{n_{1}}q_{n_{2}}\cdots
q_{n_{k}}q_{2n-1-n_{1}-n_{2}-\cdots-n_{k}}}{[\mu-(2(n-n_{1})-1)^{2}\pi
^{2}]\cdots\lbrack\mu-(2(n-n_{1}-\cdots-n_{k})-1)^{2}\pi^{2}]},\\
\delta_{m}(\mu)  &  =\sum_{n_{1},n_{2},...,n_{m+1}}\frac{q_{n_{1}}q_{n_{2}%
}\cdots q_{n_{m}}q_{n_{m+1}}(q\Phi_{n,j},e^{i(2(n-n_{1}-\cdots-n_{m+1})-1)\pi
x})}{[\mu-(2(n-n_{1})-1)^{2}\pi^{2}]\cdots\lbrack\mu-(2(n-n_{1}-\cdots
-n_{m+1})-1)^{2}\pi^{2}]}.
\end{align*}
Here, the sums are taken under the conditions $\sum\limits_{i=1}^{l}n_{i}%
\neq0,2n-1$ for $l=1,2,...,m+1$.

Therefore, the analogous formula to~\eqref{eq54} is
\begin{equation}
\biggl(\mu_{n,j}-(2n-1)^{2}\pi^{2}-\sum\limits_{k=1}^{m}\eta_{k}(\mu
_{n,j})\biggr)(\Phi_{n,j},e^{i(2n-1)\pi x})=\biggl(q_{2n-1}+\sum
\limits_{k=1}^{m}\nu_{k}(\mu_{n,j})\biggr)(\Phi_{n,j},e^{-i(2n-1)\pi x}).
\label{eq81}%
\end{equation}
Using~\eqref{eq47},~\eqref{eq80},~\eqref{eq81} and similar estimations to
those for the periodic eigenvalues, one can obtain analogous theorems to
Theorem~\ref{t7}, Theorem~\ref{t9}, and Theorem~\ref{t11} for the operator
$L_{\pi}(q)$.

Now, we present a numerical example. From the numerical results, we conclude
that, we can impose the conditions $M<2\pi^{2}$, for $n=0$, and $M<2\pi
^{2}(2n-1)$, for $n=1,2,\ldots$, where $M=\max\{|a|,b\}$,\ for the periodic
eigenvalues, instead of the conditions $M\leq\dfrac{4\pi^{2}}{3}$, for $n=0$,
and $M\leq\dfrac{4\pi^{2}(2n-1)}{3}$, for $n=1,2,\ldots$, for some specific
values of $c\in(0,1)$. Similarly, we can impose the conditions $M<4\pi^{2}$,
for $n=1$, and $M<4\pi^{2}(n-1)$, for $n=2,3,\ldots$,\ for the antiperiodic
eigenvalues, instead of the conditions $M\leq\dfrac{8\pi^{2}}{3}$, for $n=1$,
and $M\leq\dfrac{8\pi^{2}(n-1)}{3}$, for $n=2,3,\ldots$, for some specific
values of $c\in(0,1)$.

\begin{example}
For $a=-\pi^{2}$, $b=\pi^{2}$, and $c=1/2$, we have the following
approximations for the first periodic eigenvalues $\lambda_{0}$,
$\lambda_{1,1}$, $\lambda_{1,2}$, $\lambda_{2,1}$, $\lambda_{2,2}$\ and
antiperiodic eigenvalues $\mu_{1,1}$, $\mu_{1,2}$, $\mu_{2,1}$, $\mu_{2,2}$:
\begin{align*}
\lambda_{0}  &  =-0.100720167503\pi^{2},\\
\lambda_{1,1}  &  =3.953707280198\pi^{2},\qquad\lambda_{1,2}=3.976894161836\pi
^{2},\\
\lambda_{2,1}  &  =15.974913551204\pi^{2},\qquad\lambda_{2,2}%
=15.983422370241\pi^{2},\\
\mu_{1,1}  &  =0.317539742073\pi^{2},\qquad\mu_{1,2}=1.578063115969\pi^{2},\\
\mu_{2,1}  &  =8.768711027230\pi^{2},\qquad\mu_{2,2}=9.180457181326\pi^{2}.
\end{align*}
In our calculations, we take $r=s=5$. Usually it takes $8-10$ iterations with
the tolerance $1e-18$\ by the fixed point iteration method, even if we choose
an initial value that is not too close to the exact value, which means that
convergence is quite fast.

Moreover we have the following estimations for the first $4$ gaps:
\begin{align*}
\left\vert \Delta_{1}\right\vert  &  :=\mu_{1,2}-\mu_{1,1}=1.260523373896\pi
^{2}=12.440867038680\\
\left\vert \Delta_{2}\right\vert  &  :=\lambda_{1,2}-\lambda_{1,1}%
=0,023186881638\pi^{2}=0.228845349062\\
\left\vert \Delta_{3}\right\vert  &  :=\mu_{2,2}-\mu_{2,1}=0.411746154096\pi
^{2}=4.063771654597\\
\left\vert \Delta_{4}\right\vert  &  :=\lambda_{2,2}-\lambda_{2,1}%
=0.008508819037\pi^{2}=0.083978677816.
\end{align*}

\end{example}

\end{document}